%% file: main.tex
\documentclass[11pt, a4paper,
oneside, headinclude,footinclude]{scrartcl}
\usepackage[utf8]{inputenc}
\usepackage[titles]{tocloft}
\usepackage{soul}
\input{structure.tex} 
\DeclareFieldFormat{postnote}{#1}
\newcommand{\Rect}{\mathscr{R}}
\newcommand{\Regb}{\text{Reg}_{\text{b}}}
\newcommand{\Singb}{\text{Sing}_{\text{b}}}
\newcommand{\Regi}{\text{Reg}_{\text{i}}}

\newcommand{\cB}{\mathscr{B}}
\newcommand{\cU}{\mathscr{U}}
\newcommand{\cV}{\mathscr{V}}
\newcommand{\Flat}{\mathbb{F}}
\newcommand{\m}{\mathbf{m}}
\newcommand{\Phii}{\mathbf{\Phi}}

\newcommand{\AMC}{\textbf{AMC}}
\newcommand\res{\mathop{\hbox{\vrule height 7pt width .3pt depth 0pt
\vrule height .3pt width 5pt depth 0pt}}\nolimits}
\title{\normalfont\spacedallcaps{Generic uniqueness for the Plateau problem}} 

\author{\small\spacedallcaps{Gianmarco Caldini \textsuperscript{*}, Andrea Marchese \textsuperscript{*},}      \\ \small\spacedallcaps{Andrea Merlo\textsuperscript{**} and Simone Steinbr\"uchel \textsuperscript{***}}}

\date{} 

\numberwithin{equation}{section}

\begin{document}

\renewcommand{\sectionmark}[1]{\markright{\spacedlowsmallcaps{#1}}} 
\lehead{\mbox{\llap{\small\thepage\kern1em\color{halfgray} \vline}\color{halfgray}\hspace{0.5em}\rightmark\hfil}} 
\pagestyle{scrheadings}
\maketitle 
\setcounter{tocdepth}{2}

{\let\thefootnote\relax\footnotetext{* \textit{Università di Trento, Dipartimento di Matematica, Via Sommarive, 14, 38123 Povo (TN), Italy, \\e-mail:  \href{gianmarco.caldini@unitn.it, andrea.marchese@unitn.it}{gianmarco.caldini@unitn.it{\color{black},} andrea.marchese@unitn.it}.}}}
{\let\thefootnote\relax\footnotetext{** \textit{Departamento de Matem\'aticas, Universidad del Pa\' is Vasco, Barrio Sarriena s/n 48940 Leioa, Spain,\\e-mail: \href{andrea.merlo@ehu.eus}{andrea.merlo@ehu.eus}.}}}

{\let\thefootnote\relax\footnotetext{*** \textit{Institut für Mathematik, Universtät Leipzig, Augustusplatz 10, 04109 Leipzig, Germany,\\e-mail: \href{simone.steinbruechel@math.uni-leipzig.de}{simone.steinbruechel@math.uni-leipzig.de}.}}}

{\rightskip 1 cm
\leftskip 1 cm
\parindent 0 pt
\footnotesize
\textsc{Abstract}.
Given a complete Riemannian manifold $\mathcal{M}\subset\mathbb{R}^d$ which is a Lipschitz neighbourhood retract of dimension $m+n$, of class $C^{h,\beta}$ and an oriented, closed submanifold $\Gamma \subset \mathcal M$ of dimension $m-1$, which is a boundary in integral homology, we construct a complete metric space $\mathcal{B}$ of $C^{h,\alpha}$-perturbations of $\Gamma$ inside $\mathcal{M}$, with $\alpha<\beta$, enjoying the following property. For the typical element $b\in\mathcal B$, in the sense of Baire categories, there exists a unique $m$-dimensional integral current in $\mathcal{M}$ which solves the corresponding Plateau problem and it has multiplicity one. 




\par
}
{
\hypersetup{linkcolor=black, linktoc=all}
  \tableofcontents
}

\section{Introduction}\label{s:intro}
In the following let $n,m \ge 1$, $\beta \in [0,1]$ and let $\mathcal{M}\subset\R^d$ be a complete Riemannian manifold (without boundary), which is a Lipschitz neighbourhood retract\footnote{This assumption is satisfied for instance if $\mathcal{M}$ is a closed Riemannian manifold or if $\mathcal{M}=\R^{m+n}$.} of dimension $m+n$, of class $C^{h,\beta}$, with $h+\beta>3$. For every $k=0,\dots,m+n$, we denote by $\mathscr{D}_k(\mathcal{M})$ the set of $k$-dimensional currents (or $k$-currents) with support in $\mathcal{M}$ and by $\mathscr{I}_k(\mathcal{M})$ the subgroup of $k$-dimensional integral currents. We refer to Section \ref{s:notation} for the relevant definitions. We denote by $\mathbf{AMC}(b)$ the set of area-minimizing integral currents in $\mathcal{M}$ with boundary $b$, namely
$$\AMC(b):=\{T\in\mathscr{I}_m(\mathcal{M}):\partial T=b,\, \Mass(T)\leq\Mass(S)\mbox{ for every $S\in\mathscr{I}_m(\mathcal{M})$ with $\partial S=b$}\}.$$
We denote the set of $(m-1)$-dimensional boundaries in $\mathcal{M}$ by $$\mathscr{B}_{m-1}(\mathcal{M}):=\{b \in \mathscr{D}_{m-1}(\mathcal{M}) : b=\partial T\mbox{ for some }T\in\mathscr{D}_m(\mathcal{M})\}.$$

Let $\Gamma \subset \mathcal{M}$ be an oriented, closed (\textit{i.e.} compact and without boundary) submanifold of dimension $m-1$ and of class $C^{\ell,\alpha}$, with $3<\ell+\alpha<h+\beta$. Let $b_0:=\llbracket \Gamma \rrbracket$ be the associated  multiplicity-one current and assume that $b_0\in\mathscr{B}_{m-1}(\mathcal{M})$.
For every $P \in \Gamma$ there exists a connected, open set $U\subset \R^{m+n}$, a diffeomorphism $\Phii:U\to\Phii(U)\subseteq\mathcal{M}$ of class $C^{h,\beta}$ such that $P\in\Phii(U)$, a relatively open, connected, bounded set $\Omega\subset\R^{m-1}=\langle e_1,\dots,e_{m-1}\rangle$, and a function $f : \Omega \rightarrow \R^{n+1}$ of class $C^{\ell,\alpha}$ such that $$gr(f):=\{(x,y)\in\Omega\times\langle e_m,\dots,e_{m+n}\rangle: y=f(x)\}$$
satisfies $gr(f)\subset U$ and such that
\begin{equation}\label{e:bizzero}
\Gamma\cap\Phii(U)=\Phii(gr(f)).
\end{equation} 
Observe that since $\Omega$ is connected, then \eqref{e:bizzero} implies that $\Gamma\cap \Phii(U)$ is also connected.\\

Given a connected open set $\Omega'$ compactly contained in $\Omega$ and $\varepsilon>0$, we let 
\begin{equation}\label{e:deficseps}
    X_{\varepsilon}(P):=\{u\in C^{\ell,\alpha}(\Omega, \R^{n+1}): f-u\equiv 0 \mbox { on } \Omega\setminus\Omega', \lVert f-u\rVert_{C^{\ell,\alpha}}\leq \varepsilon\}.
\end{equation}
By \eqref{e:bizzero} there exists $\varepsilon>0$ such that
\begin{equation}\label{e:sceltaeps}
    \text{$gr(u)\subseteq U$ for every $u\in X_\varepsilon(P)$.}
\end{equation}
We endow $X_\varepsilon(P)$ with the norm $\lVert\cdot\rVert_{C^{\ell,\alpha}}$, which makes it a complete metric space, see Lemma \ref{lemmachiusura}.\\

For $i=1,\dots,N$, we select one point $p_i$ on each connected component of $\Gamma$ and we assume that the definition of $U_i,\Phii_i, \Omega_i$, $f_i$ and $\varepsilon_i$ as in \eqref{e:sceltaeps} is understood. We assume that $\Phii_i(U_i)$ are disjoint and we denote 
\begin{equation}\label{e:defeta}
\eta:=\min\{1;\min_{i=1,\dots,N}\varepsilon_i\}.    
\end{equation}
Further restrictions on $\eta$ will be specified in Lemma \ref{l:denseimpliesresidual}. We denote by ${\mathbf{X}}_\eta$ the product space
\begin{equation}\label{e:defX}
{\mathbf{X}}_\eta:=\prod_{i=1}^N X_{\eta}(p_i),
\end{equation}
endowed with the 1-product distance, namely the distance induced by the norm
\begin{equation}\label{e:proddist}
\|(u_1,\dots,u_N)\|:=\sum_{i=1}^N\|u_i\|_{C^{\ell,\alpha}(\Omega_i)}.
\end{equation}
We define a map $\Psi:{\mathbf{X}}_\eta\to\mathscr{B}_{m-1}(\mathcal{M})$ as follows 
\begin{equation}\label{e:defPSI}
    \Psi(u_1,\dots,u_N):= \sum_{i=1}^N\llbracket\Phii_i(gr(u_i)) \rrbracket + b_0\res (\mathcal{M}\setminus \bigcup_{i=1}^N\Phii_i(U_i)).
\end{equation}
We observe that $\Psi$ is injective and  $\Psi(u_1,\dots,u_N)$ and $b_0$ are in the same homology class for every $(u_1,\dots,u_N)\in {\mathbf{X}}_\eta$, see Lemma \ref{l:cobordanti}.
We define the space of boundaries associated to ${\mathbf{X}}_\eta$ as
\begin{equation}\label{e:defB}
\mathcal{B}_\eta:=\Psi({\mathbf{X}}_\eta).    
\end{equation}
We naturally endow $\mathcal{B}_\eta$ with the distance $d$ induced by the map $\Psi$. More precisely, for every $b\in {\mathbf{X}}_\eta$ we denote \begin{equation}\label{e:inverse}
    (u_1(b),\dots, u_N(b)):=\Psi^{-1}(b)
\end{equation} and we define 
\begin{equation}\label{e:distance}
d(b,\bar b):=\sum_{i=1}^N\|u_i(b)-u_i(\bar b)\|_{C^{\ell,\alpha}(\Omega_i)}.    
\end{equation}
Notice that $(\mathcal{B}_\eta, d)$ is also a complete metric space, because $\Psi$ is by definition an isometry,  see Lemma \ref{lemmachiusura}. Roughly speaking, the space $\mathcal{B}_\eta$ consists of $C^{\ell,\alpha}$-perturbations of the boundary $\Gamma$ that allow us to deform each connected component of $\Gamma$, locally around a point. Observe that \eqref{e:sceltaeps} implies that every boundary in $\mathcal{B}_\eta$ is the multiplicity-one current associated to an oriented regular submanifold.
We are ready to state the main results of this paper which we prove in Section \ref{s:main_proofs}.

\begin{teorema}\label{t:main}  
For the typical boundary $b \in \mathcal{B}_\eta$, any area-minimizing integral current $T$ with $\partial T=b$ has multiplicity one $\|T\|$-a.e.
\end{teorema}


In codimension $n=1$ the previous theorem has the following interesting consequence.

\begin{corollario}\label{t:codim1}  
If $n=1$, then for the typical boundary $b \in \mathcal{B}_\eta$, any area-minimizing integral current $T$ with $\partial T=b$ has density $1/2$ at every point of the support of $b$.
\end{corollario}

We also deduce the following general result.

\begin{teorema}\label{c:main}
For the typical boundary $b \in \mathcal{B}_\eta$, there is a unique area-minimizing integral current $T$ with $\partial T=b$.
\end{teorema}

Since the intersection of residual sets is a residual set, then for the typical boundary $b \in \mathcal{B}_\eta$
both the conclusion of Theorem \ref{t:main} and the conclusion of Theorem \ref{c:main} are satisfied.

 In Section \ref{s:flat}, we obtain a result in the spirit of Theorem \ref{c:main}, replacing the space $\mathcal{B}_\eta$ with a larger space of boundaries. On the other hand, the strong norm considered on $\mathcal{B}_\eta$ needs to be naturally substituted by a weaker one and we work on the manifold $\mathcal{M}:=\R^{m+n}$, with $m>1$.

We fix an arbitrary $C>0$, a compact, convex set $K \subset \R^{m+n}$ with nonempty interior and define \begin{equation} \mathcal{R}_C:=\{b \in \mathscr{B}_{m-1}(K) \cap \mathscr{I}_{m-1}(K): \mathbb{M}(b)\leq C\}.\end{equation}
We metrize $\mathcal{R}_C$ with the distance $d_{\,\flat}$ induced by the flat norm, see \eqref{e:flatdef}.

Also in this case, the space is complete, see Lemma \ref{l3:closed}, and we obtain the analogous result to Theorem \ref{c:main}, that is, the following 

\begin{teorema}\label{t:genericga}
For the typical boundary $b \in \mathcal{R}_C$, the set $\mathbf{AMC}(b)$ is a singleton.
\end{teorema}

Even if the previous result is arguably weaker than Theorem \ref{c:main}, the strategy is quite flexible and can be adapted to variational problems in which the singular set of minimizers is so large that it can disconnect the regular part, see \cite{genericuniqbranchedtransport}.

\subsection{Content of the paper}

In Section \ref{s:notation}, we introduce the notation for currents and prove preliminary properties of our space of boundaries. In Section \ref{s:the_game} we play a Banach-Mazur game in the following context. The main idea behind Theorem \ref{t:main} is that for an area minimzing integral current, regular two-sided boundary points are contained in the support of the current which locally is smooth. However, the typical boundary is not contained in any 
$m$-dimensional submanifold of higher regularity (proof of this can be found in Section \ref{s:the_game}) and thus, the typical boundary does not allow for area-minimizing currents with regular two-sided boundary points, which we prove in Section \ref{s:main_proofs}.
We deduce Theorem \ref{c:main} exploiting a technique introduced in \cite{Morganinventiones}. In Section \ref{s:flat} we prove Theorem \ref{t:genericga}.


\subsection{Previous results on generic properties of area-minimizing currents}

\emph{Generic} properties, in the sense of Baire categories, are of fundamental importance in the study of the well-posedness of solutions to geometric variational problems. Fine results have been derived when the ambient manifold is endowed with a $C^{\infty}$-generic metric, such as density, equidistribution, multiplicity one and Morse index estimates of min-max minimal hypersurfaces, see \cite{IrieMarNev, MarNevSong, ZhouM1Conj, MarNev21}. Recent generic regularity results have been obtained for locally stable minimal hypersurfaces in 8-dimensional closed Riemannian manifolds and for minimizing hypersurfaces in ambient manifolds of dimension 9 and 10, see \cite{WangLi22, ChodoshLioSpolaorARS} and \cite{ChodoshMantouSchulze} respectively. In other words, it has been shown that singularities can be ``perturbed away'' for generic ambient metrics or for slight perturbations of the boundary, leading to generic smoothness of solutions. Generic regularity for higher dimensional hypersurfaces is still an open problem and not much is known about generic regularity of minimal submanifolds with codimension higher than one, see \cite{White85HGeneric, White19HGeneric}.

Another question occurring naturally in connection with the Plateau problem is that of uniqueness of solutions: it goes back at least to the first decades of the twentieth century, to works by many authors, see \cite{Courantbook, DierkesbookI, Douglas, Nitsche, Rado}. There are many examples of curves admitting several different minimizers, see \cite{Morgancontinuum, Nitsche3}. However, the presence of many symmetries motivated the question whether uniqueness is a generic property itself, see \cite[Section I.11, (3)]{Bigregularity}.

Morgan proved in \cite{Morganinventiones} that almost every curve in $\R^3$ (with respect to a suitable measure) bounds a unique area-minimizing surface. The result has been later generalized by the same author to elliptic integrands and to any dimension and, in the special case of area-minimizing flat chains modulo 2, to any codimension, see \cite{Morganindiana, Morganarma}.  We remark that all Morgan's results are restricted to the Euclidean ambient setting or rely on uniform convexity assumptions on the boundary. In fact his proofs depend on Allard's boundary regularity theorem for stationary varifolds, which states that if a boundary $\Gamma$ is supported in the boundary of a uniformly convex set, then every point $p \in \Gamma$ has density $1/2$ and it is regular, see \cite{Allardphd}, \cite[\S 4]{Allardboundary}, \cite[Proposition 6.1]{Morganinventiones} and \cite[\S 4]{Morganindiana}. This allows Morgan to rule out the existence of \textit{two-sided} regular boundary points, namely regular boundary points at which the current ``crosses'' the boundary, see \cite[Example 1.3]{Borda}.  

Hardt and Simon proved in \cite{hardtSimonBoundary} that, for codimension one area-minimizing currents in the Euclidean space, every boundary point is regular (possibly two-sided) without assuming the convexity condition. More recently, the fourth-named author extended this result to codimension one Riemannian ambient manifolds, see \cite{steinbruechel}. A recent result by De Lellis, De Philippis, Hirsch and Massaccesi, see \cite{Borda}, proves the first general boundary regularity theorem with no restrictions on the codimension, showing that the set of regular boundary points (possibly two-sided) is dense, see also \cite{Nardulli4president} for a 2-dimensional analogue allowing for arbitrary boundary multiplicity.\\

In this article we prove generic uniqueness and the multiplicity-one property of area-minimizing integral currents in full generality, \emph{i.e.} for general ambient manifolds $\mathcal{M}$ of any dimension, for any codimension and with no convexity assumption on the geometry of the boundary $\Gamma$. Our result relies on the aforementioned boundary regularity theorem in \cite{Borda}, as our main goal is to prove the generic absence of two-sided boundary points.  It is worth mentioning that Morgan hints at generic uniqueness in the Baire sense when the ambient space is a manifold (see \cite[Remark 5.4]{Morganindiana}), but this case is restricted to codimension one submanifolds and still under convexity assumptions on the boundary.

\subsection*{Acknowledgements.}
We would like to thank Camillo De Lellis, Guido De Philippis, and Emanuele Spadaro for valuable discussions. The work of A. Ma. is partially supported by PRIN 2017TEXA3H\_002 "Gradient flows, Optimal Transport and Metric Measure Structures" and by GNAMPA-INdAM.  
During the writing of this work A.Me. 
was partially supported by the European Union’s Horizon Europe research and innovation programme under the Marie Sk\l odowska-Curie grant agreement no 101065346.
S.S. is supported by the German
Science Foundation DFG in context of the Priority Program SPP 2026 “Geometry at Infinity”.



\section{Notation and preliminaries}\label{s:notation}
We briefly recall the relevant definitions of the theory of currents and we refer the reader to \cite{Federer1996GeometricTheory, Simon1983LecturesTheory} for a complete treatment of the subject. A $k$-dimensional \emph{current} on $\R^d$ ($k \leq d$) is a continuous linear functional on the space $\D^k(\R^d)$ of smooth and compactly supported differential $k$-forms in $\R^d$. The space of $k$-dimensional currents in $\R^d$ is denoted by $\D_k(\R^d)$. The \emph{boundary} of a current $T\in\D_k(\R^d)$ is the current $\partial T\in\D_{k-1}(\R^d)$ such that $$\partial T(\varphi)=T(d\varphi), \quad\mbox{ for every $\varphi\in\D^{k-1}(\R^d)$},$$
where as usual $d$ denotes the exterior differential.
Given $T \in \D_k(\R^d)$, the \emph{mass} of $T$ is denoted by $\Mass(T)$ and is defined as the supremum of $T(\omega)$ over all forms $\omega$ with $|\omega(x)| \leq 1$ for all $x \in \R^d$. The \emph{support} of a current $T$, denoted $\text{supp}(T)$, is the intersection of all closed sets $C$ in $\R^d$ such that $T(\omega)=0$ whenever $\omega \equiv 0$ on $C$. For every closed subset $K$ of $\R^d$, we will denote by $\D_k(K)$ the set 
$$\mathscr{D}_k(K):=\{T \in \mathscr{D}_k(\R^d) \mid \operatorname{supp}(T) \subset K\}.$$

Given a smooth, proper map $f : \R^d \rightarrow \R^{d'}$ and a $k$-current $T$ in $\R^d$, the \emph{push-forward} of $T$ according to the map $f$ is the $k$-current $f_{\sharp}T$ in $\R^{d'}$ defined by
\begin{equation}\label{d:push-forward}
f_{\sharp}T (\omega) := T(f^{\sharp}\omega), \quad\mbox{for every $\omega\in\mathscr{D}^k(\R^{d'})$},\end{equation} 
where $f^{\sharp}\omega$ denotes the pullback of $\omega$ through $f$. If $T$ has finite mass and compact support, then the previous definition can be extended to any $f$ of class $C^1$.

We say that a current $T\in\D_k(\R^d)$ is \emph{integer rectifiable} and we write $T\in\Rect_k(\R^d)$ if we can identify $T$ with a triple $(E,\tau,\theta)$, where $E\subset K$ is a $k$-rectifiable set, $\tau(x)$ is a unit $k$-vector spanning the tangent space $T_xE$ at $\Haus^k$-a.e. $x$ and $\theta\in L^1(\Haus^k \res E, \mathbb{Z})$ is an integer-valued multiplicity, where the identification means that the action of $T$ can be expressed by
\begin{equation}\label{e:rectifiablecurrent}
T(\omega)=\int_E\langle\omega(x),\tau(x)\rangle \theta(x) d\Haus^k(x), \quad \mbox{ for every $\omega\in \D^k(\R^d)$}.\end{equation}
If $T$ is as in \eqref{e:rectifiablecurrent}, we denote it by $T=\llbracket E, \tau, \theta \rrbracket$. We denote by $\mathscr{I}_k(\R^d)$ the subgroup of $k$-dimensional \emph{integral currents}, that is the set of currents $T\in\Rect_k(\R^d)$ with $\partial T\in\Rect_{k-1}(\R^d)$. If $T=\llbracket E, \tau, \theta \rrbracket\in\Rect_k(\R^d)$ and $B \subset \mathbb{R}^d$ is a Borel set, we denote the \emph{restriction} of $T$ to $B$ by setting $T\res B:=\llbracket E \cap B, \tau, \theta \rrbracket.$
The set of integer rectifiable (respectively integral) $k$-currents with support in a closed set $K$ is denoted by $\Rect_k(K)$ (respectively $\mathscr{I}_k(K)$).\\

We recall that the \emph{(integral) flat norm} $\Flat(T)$ of an integral current $T \in \mathscr{I}_k(K)$, with $K$ compact, is defined by:
\begin{equation}\label{e:flatdef}
\Flat(T) := \min \{ \Mass(R) + \Mass(S) \mid T=R +\partial S, \, R \in \mathscr{I}_k(K),\, S \in \mathscr{I}_{k+1}(K)\}.
\end{equation}
A $k$-dimensional \emph{polyhedral} current is a current $P$ of the form  
\begin{equation}\label{e:poly}
P:=\sum_{i=1}^N\theta_i\llbracket \sigma_i\rrbracket,
\end{equation}
where $\theta_i\in \R$, $\sigma_i$ are non-overlapping $k$-dimensional simplexes in $\R^d$, oriented by (constant) $k$-vectors $\tau_i$ and $\llbracket \sigma_i \rrbracket=\llbracket\sigma_i,\tau_i,1\rrbracket$ is the multiplicity-one current naturally associated to $\sigma_i$. A polyhedral current with integer coefficients $\theta_i$ is called \emph{integer polyhedral} and we denote the subgroup of $k$-dimensional integer polyhedral currents with support in $K$ by $\mathscr{P}_k(K)$.

\begin{lemma}\label{l:flatequivalence}
There exists a constant $C>0$ (depending only on $\max_i\{\Mass(\llbracket\Omega_i\rrbracket)\}$ and $\max_i\{\mathrm{Lip}(\mathbf{\Phi}_i)\}$) such that $\Flat(b-\bar b)\leq C \, d(b,\bar b)$, for every $b,\bar b\in \mathcal{B}_\eta$.
\end{lemma}

\begin{proof}
It is sufficient to prove the lemma for $N=1$. Indeed, denoting for every $b\in\mathcal{B}_\eta$ and for $i=1,\dots, N$ the boundary $b^i\in\mathcal{B}_\eta$ defined by $$b^i:=b\res(\Phii_i(U_i))+\llbracket\Gamma\rrbracket\res (\mathcal{M}\setminus \Phii_i(U_i)),$$
we have $$\bar b-b=\sum_{i=1}^N \bar b^i-b^i,$$ so that
$$\Flat(\bar b- b)\leq\sum_{i=1}^N\Flat(\bar b^i-b^i)\leq N\max_{i=1,\dots,N}\Flat(\bar b^i- b^i).$$
Hence we can assume that $N=1$ and for $w\in {\mathbf{X}}_\eta$ we define $\textbf{w}:\Omega\to\R^{m+n}$ by 
\begin{equation}\label{e:defbold}
\textbf{w}(x):=(x,w(x)).    
\end{equation}  
Let $u:=\Psi^{-1}(b)$ and $\bar u:=\Psi^{-1}(\bar b)$
and we denote $I:=\llbracket [0,1] \rrbracket \in \mathscr{I}_1(\R)$ and we let $F:[0,1] \times \Omega \rightarrow \R^{m+n}$ be the linear homotopy $$F(t,x)=(1-t)\textbf{u}(x)+ t\bar{\textbf{u}}(x).$$ Denote $S:=F_{\sharp}(I \times \llbracket 
\Omega \rrbracket)$. We use \cite[26.18]{Simon1983LecturesTheory} to compute 
\begin{equation*}
    \begin{split}
    \partial S=&F_\sharp(\partial(I\times\llbracket \Omega \rrbracket))=F_\sharp(\partial I\times\llbracket \Omega \rrbracket-I\times\partial\llbracket \Omega \rrbracket)\\
    =& F_\sharp(\delta_1\times\llbracket \Omega \rrbracket)-F_\sharp(\delta_0\times\llbracket \Omega \rrbracket)-F_\sharp(I\times\partial\llbracket \Omega \rrbracket)\\
    =&(\bar{\textbf{u}})_\sharp\llbracket \Omega \rrbracket-(\textbf{u})_\sharp\llbracket \Omega \rrbracket -F_\sharp(I\times\partial\llbracket \Omega \rrbracket)\\
    =& \llbracket gr(\bar u) \rrbracket-\llbracket gr(u) \rrbracket-F_\sharp(I\times\partial\llbracket \Omega \rrbracket)=\llbracket gr(\bar u) \rrbracket-\llbracket gr(u) \rrbracket,
\end{split}
\end{equation*}
where the last equality is due to the fact that $\bar{\textbf{u}}=\textbf{u}$ on $\partial\Omega$. Hence, by the homotopy formula, {see \cite[\S 4.1.9]{Federer1996GeometricTheory}, we can estimate
\begin{equation} \label{e:homotopy_estimate}
\begin{aligned} 
    \Flat(\llbracket gr(\bar u) \rrbracket-\llbracket gr(u) \rrbracket)& \leq \lVert \bar{\textbf{u}} - \textbf{u} \rVert_{\infty} \operatorname{sup}_{x \in \Omega}\left(|D\textbf{u}(x)-D\bar{\textbf{u}}(x)|\right)^{m-1} \Mass(\llbracket \Omega \rrbracket) \\ 
    &\leq C \lVert \bar{\textbf{u}} - \textbf{u} \rVert_{C^{\ell,\alpha}}=C \lVert \bar u - u \rVert_{C^{\ell,\alpha}},
\end{aligned}
\end{equation}
where the constant $C$ in the second line depends only on $m$ and the Lipschitz constant of $\bar u-u$, which is bounded by 2 by definition of ${\mathbf{X}}_\eta$}. Therefore, since $\Phii$ is of class $C^{h,\beta}$, we infer
$$\Flat(\bar b- b)=\Flat(\Phii_\sharp(\llbracket gr(\bar  u) \rrbracket-\llbracket gr(u) \rrbracket))
\leq C\, \mathrm{Lip}(\mathbf{\Phi}_i)^{m-1}\lVert \bar u - u \rVert_{C^{\ell,\alpha}}=C\, d(\bar b, b),$$
where the last identity follows from the definition of the distance $d$. 
\end{proof}

\begin{lemma}\label{l:cobordanti}
For every $b\in\mathcal{B}_\eta$ there exists a current $S\in\mathscr{I}_m(\mathcal{M})$ such that $\llbracket\Gamma\rrbracket-b=\partial S$. In particular all the elements of $\mathcal{B}_\eta$ are in the same homology class.
\end{lemma}
\begin{proof}
For every connected component of $\Gamma$, we consider the corresponding $U_i,\Phii_i, \Omega_i$, $f_i$, defined in the introduction. We now argue as in the proof of Lemma \ref{l:flatequivalence}, replacing $\bar{b}$ with $\llbracket\Gamma\rrbracket$ to define a current $S_i\in\mathscr{I}_m(\R^{m+n})$ such that $\partial S_i=\llbracket gr(f_i)\rrbracket - \llbracket gr(u_i)\rrbracket$.
The current $S:=\sum_{i=1}^N(\Phii_i)_{\sharp}(S_i)$ satisfies the requirement.
\end{proof}

\section{The typical $C^{\ell,\alpha}$ graph avoids $C^{h,\beta}$ submanifolds}\label{s:the_game}


The proof of Theorem \ref{t:main} is obtained combining the boundary regularity result of \cite{Borda} and the following property of the typical map $u\in X_\varepsilon(P)$, see \eqref{e:deficseps}. For every open set $V\subset U\subset\R^{n+m}$ such that $gr(u\res\Omega')\cap V\neq\emptyset$ and, for every $m$-dimensional submanifold $\mathcal{N}$ of class $C^{h,\beta}$ in $\R^{m+n}$ with $\partial\mathcal{N}\cap V=\emptyset$ it holds $gr(u)\cap V\not\subset\mathcal{N}$.

For the sake of generality, in this section we prove this result for $u:\R^{m-k}\to\R^{n+k}$, for every $k<m$. 
For the purpose of this paper, this generalization is unnecessary, however, we include it since it does not require any additional effort.\\

In the following let $n,m\geq 1$, and $0\leq k<m$.
Throughout this section we will denote $\{e_1,\dots, e_{m+n}\}$ the standard basis of $\R^{m+n}$. Let $\Omega$ be a fixed open bounded set in $\R^{m-k}=\langle e_1,\dots, e_{m-k}\rangle$. We further fix $h\in\N \setminus{\{0\}}$, $\ell\in\N$, $\alpha,\beta\in[0,1]$ and $\gamma\in[0,2]\setminus\{1\}$ so that $\ell+\alpha<\ell+\gamma<h+\beta$, a function $f:\Omega\to \R^{n+k}$ of class $C^{\ell,\alpha}$ and an open set $\Omega'$ compactly contained in $\Omega$. For fixed $\varepsilon>0$, we let 
\begin{equation*}
    X_{\varepsilon}:=\{u\in C^{\ell,\alpha}(\Omega, \R^{n+k}): f-u\equiv 0 \mbox{ on }\Omega\setminus\Omega',\, \lVert f-u\rVert_{C^{\ell,\alpha}}\leq \varepsilon\},
\end{equation*}
where we denoted
$$\lVert u\rVert_{C^{\ell,\alpha}}=\lVert u\rVert_{C^\ell}+[D^\ell u]_\alpha:=\lVert u\rVert_\infty+\sum_{j=1}^\ell\lVert D^ju\rVert_\infty+\sup_{x\neq y\in\Omega} \frac{|D^\ell u(x)-D^\ell u(y)|}{\lvert x-y\rvert^\alpha}.
$$
we further endow $X_\varepsilon$ with the norm $\lVert\cdot\rVert_{C^{h,\alpha}}$. We observe that the space $X_\varepsilon(P)$ defined in \eqref{e:deficseps}, fits this definition with $k=1$ and $\ell\geq 3$.\\

We begin with the following observation.

\begin{lemma}\label{lemmachiusura}
The space $(X_\varepsilon,\lVert\cdot\rVert_{C^{\ell,\alpha}})$ is complete. In particular the space $(\mathcal{B}_\eta,d)$ is also complete.
\end{lemma}
\begin{proof}
It suffices to show that $X_\varepsilon$ is closed in $(C^{\ell,\alpha}, \lVert\cdot\rVert_{C^{\ell,\alpha}})$. Let $u_n$ be a sequence of elements in $X_\varepsilon$ and let $u \in C^{\ell,\alpha}$ be such that $\lVert u_n - u \rVert_{C^{\ell,\alpha}} \rightarrow 0$. Obviously $f-u\equiv 0$ on $\Omega\setminus\Omega'$ and $u\in C^{\ell,\alpha}$, hence $u \in X_{\varepsilon}$.

The fact that $\mathcal{B}_\eta$ is complete follows from the fact that $\Psi$ is an isometry between the product space { ${\mathbf{X}}_\eta$} defined in \eqref{e:defX} endowed with the distance induced by the norm \eqref{e:proddist} and $(\mathcal{B}_\eta,d)$.
\end{proof}
We then introduce a subset of $X_\varepsilon$ which roughly consists of those functions whose graph has small intersection with any submanifold of class $C^{h,\beta}$. We let $\pi_\Omega:\Omega\times{\R^{n+k}}\to\Omega$ be the orthogonal projection on the first $m-k$ coordinates of $\R^{m+n}$.
For every open set $A\subset\Omega$ we denote 
$$C_A:=\{(z_1,z_2)\in \Omega\times \R^{n+k}:z_1\in A\}$$ 
and we abbreviate $C(x,r):=C_{B(x,r)}$. 

\begin{definizione}\label{def:bordibrutti}
Let $\mathcal{A}$ be the set of those $ w\in X_\varepsilon$ for which there exists an embedded $m$-dimensional manifold $\mathcal{N}\subset\R^{m+n}$ of class $C^{h,\beta}$ and an open set $O\subset C_{\Omega'}$ such that 
\begin{equation*}
    \partial \mathcal{N}\cap O=\emptyset\qquad \text{and}\qquad\emptyset\neq \mbox{gr}(w)\cap O\subset\mathcal{N}.
\end{equation*}
\end{definizione}

The aim of this section is to prove the following proposition.

\begin{proposizione}\label{p:baire30}
The set $\mathcal{A}$ is of first category in $X_\varepsilon$, i.e. it is contained in a countable union of closed sets with empty interior.
\end{proposizione}

Thanks to Lemma \ref{lemmachiusura} and Baire's theorem, Proposition \ref{p:baire30} implies in particular that $X_\varepsilon\setminus\mathcal{A}$ is dense in $X_\varepsilon$.
Our strategy to prove Proposition \ref{p:baire30} uses the relation between topological properties of sets in the sense of Baire categories and the existence of a winning strategy for a suitable topological game. Let us quickly recall such general result.

\begin{definizione}[Banach-Mazur game]
Let $(X,\mathcal{T})$ be a topological space and let $A\subseteq X$ be an arbitrary subset. The \emph{Banach-Mazur game} associated to $A$ is a game between two players, $P1$ and $P2$ with the following rules: $P1$ chooses arbitrarily a { non-empty} open set $\cU_1\subseteq X$; then $P2$ chooses an open set $\cV_1\subseteq \cU_1$; then $P1$ chooses a { non-empty} open set $\cU_2\subseteq \cV_1$ and so on.
If the set $\left(\bigcap_{i\in\N} \cV_i\right)\cap A$ is non-empty then $P1$ wins. Otherwise $P2$ wins.\label{BMgame}
\end{definizione}

The following proposition relates the Banach-Mazur game to the topology of the space on which it is played. We say that a set is of \emph{first category} if it is contained in a countable union of closed subsets with empty interior. A set is \emph{residual} if its complement is of first category. We say that a certain property holds for the \emph{typical} element of $X$, if it holds for every element of a residual set. We say that $P2$ has a \emph{winning strategy} if, for every choice of open sets $\mathscr U_i$ by $P1$, there exists a choice of open sets $\mathscr V_i$ for $P2$ such that 
$\Big(\bigcap_{i\in\N}\mathscr V_i\Big)\cap A=\emptyset.$

\begin{proposizione}
Suppose the metric space $X$ is complete. Then there exists a winning strategy for $P2$ if and only if $A$ is of first category in $X$.\label{BMgioco}
\end{proposizione}

\begin{proof}
The proof of this result is given in~\cite{Oxtobi} only in the case of the real line. However the same argument works verbatim in any complete metric space.
\end{proof}

\begin{definizione}\label{insiemeA}
Let $A$ be the set of those $ w\in X_\varepsilon$ for which there exists a map $M:\Omega\times \langle e_{m-k+1},\ldots, e_m \rangle\to \R^{n}$ of class $C^{h,\beta}$ and an open set $W\subset\Omega'$ such that $$\pi_\Omega(C_W\cap \mathrm{gr}(M)\cap \mathrm{gr}( w))= W.$$
\end{definizione}

The main step for the proof of Proposition \ref{p:baire30} is the following

\begin{proposizione}\label{p:baire1}
The set $A$ is of first category in $X_\varepsilon$.
\end{proposizione}

We postpone the proof to the end of the section and begin with the following elementary lemma.

\begin{lemma}
\label{prop:espansioneChbeta} 
Suppose that $g:\R\to \R$ is a function of class $C^{h,\beta}$. Then, for any $x\in \R$ there exists a $t_0>0$ and a bounded function $r_{h+1}(t):[-t_0,t_0]\to \R$ such that for every $t\in[-t_0,t_0]$ $$g(x+t)=g(x)+d g(x)t+\ldots+\frac{d^h g(x)}{h!}t^h+r_{h+1}(t)t^{h+\beta}.$$
In addition $\lVert r_{h+1}(t)\rVert_\infty\leq \frac{[d^h g(x)]_{C^\beta}}{h!}$.
\end{lemma}

\begin{proof}
The Taylor expansion of $g$ yields
\begin{equation}
 g(x+t)=g(x)+ dg(x)t+\ldots+\frac{d^{h-1}g(x)}{(h-1)!}t^{h-1}+\frac{d^hg(x+\zeta_t)}{h!}t^{h},
\end{equation}
for some $\zeta_t\in [0,t]$. However, this shows that
\begin{equation}
  g(x+t)-g(x) - dg(x)t - \ldots - \frac{d^{h-1}g(x)}{(h-1)!}t^{h-1} - \frac{d^hg(x)}{h!}t^h=\frac{d^hg(x+\zeta_t)t^h-d^hg(x)t^h}{h!}.
\end{equation}
Define $r_{h+1}(t):=\frac{ d^hg(x+\zeta_t)-d^hg(x)}{h!t ^\beta}$ 
and note that $|r_{h+1}(t)|\leq\frac{|d^hg(x+\zeta_t)-d^hg(x)|}{h!\zeta_t ^\beta}\leq \frac{[d^h g(x)]_{C^\beta}}{h!}$.
\end{proof}

The following proposition provides the main tool to find a winning strategy for the Banach-Mazur game associated to $A$, allowing us to prove Proposition \ref{p:baire1}. We denote the balls in $X_\eps$ by $\cB(w, \rho) = \{ u \in X_\eps: \lVert u-w \rVert_{C^{\ell,\alpha}}< \rho\}$.

\begin{proposizione}\label{lemma:principale}
Let $\bar w\in X_\varepsilon$ be fixed and let $\bar\rho>0$, $j\in\N\setminus\{0\}$. Then, for any $x\in \Omega'$ there exist $u\in X_\varepsilon$ and $\rho>0$ such that 
\begin{itemize}
    \item[(i)] $\cB(u, \rho)\subseteq \cB(\bar w,\bar\rho)$; 
    \item[(ii)] for every $w\in \cB(u,\rho)$ and $M:\Omega\times \langle e_{m-k+1},\ldots, e_m \rangle\to \R^{n}$ of class $C^{h,\beta}$ with $\lVert M\rVert_{C^{h,\beta}}\leq j$ we have
    \begin{equation}
            \pi_\Omega(\mathrm{gr}(M)\cap \mathrm{gr}( w)\cap C(x, r))\neq B(x, r),
            \label{eq:biginterior}
    \end{equation}
where $r:=\min\{1/j, \mathrm{dist}(x,\Omega\setminus\Omega')\}$.\end{itemize}
\end{proposizione}



\begin{proof}
Assume by contradiction that there is an $x\in\Omega'$ such that for every $u\in \cB(\bar w,\bar\rho)$ there is an infinitesimal sequence $\rho_i\leq \bar\rho-\lVert u-\bar w\rVert_{C^{\ell,\alpha}}$ for which property (ii) fails.

In the following we construct a function $u\in X_\varepsilon$, which is a small perturbation of $\bar w$ in the $C^{\ell,\alpha}$ norm, by adding a bump of class $C^{\ell,\gamma}$ (along a specific direction) to a mollified version of $\bar w$ in such a way that the optimal regularity of the perturbation $u$ is $C^{\ell,\gamma}$. Via Taylor expansion, we show that the bound on the $C^{h,\beta}$ norm of $M$ implies that \eqref{eq:biginterior} holds with $u$ in place of $w$. By a simple compactness argument, this implies that (ii) holds for $\rho$ sufficiently small: a contradiction.

Fix $0<\delta<1$ to be chosen later and let $\psi_\delta$ be the function on $\R^{m-k}$ defined by
$$\psi_\delta(z):=\delta^{1+\gamma}|x_1-z_1|^{\ell+\gamma},$$
where $z_i$ are the coordinates of $z\in \R^{m-k}$. Let $r$ be as in (ii). Let $\eta:\Omega\to [0,1]$ be a smooth cutoff function such that $\eta\equiv0$ on $\Omega\setminus B(x,r)$ and $\eta\equiv 1$ on $B(x,r/2)$. Observe that $\eta\psi_\delta \in C^{\ell, \alpha}$ and more precisely \begin{equation}
    \lim_{\delta\to 0}\delta^{-1}\|\eta \psi_\delta\|_{C^{\ell,\alpha}}
    =0\qquad\text{ and }\qquad\lvert \psi_\delta(x+te_1)\rvert
    =\delta^{1+\gamma}\lvert t\rvert^{\ell +\gamma} 
    \qquad\text{for any $\lvert t\rvert \leq r/2$}.
    \label{e:normapsi}
\end{equation}

Throughout the rest of the proof, we let $\varphi$ be a mollification kernel, that is a non-negative, radial smooth function supported on $B(0,1)\subseteq \R^{m-k}$ such that $\varphi\equiv 1$ on $B(0,1/2)$ and $\int \varphi=1$. For any $\iota \in \N \setminus\{0\}$ we further let $\varphi_{\iota}(y):=\iota^{m-k}\varphi(\iota y)$. Denote $f_\delta:=(1-\delta)\bar w+\delta f$ and define
$$v:=\varphi_\iota*(\eta f_\delta)+(1-\eta) f_\delta.$$
Observe that for $\iota$ sufficiently large we have that $f-v\equiv 0$ on $\Omega\setminus\Omega'$. Moreover, the function $v$ is smooth on $B(x,r/2)$. Denote 
$$u:=(v_1, \dots, v_{n+k-1}, v_{n+k}+\eta\psi_\delta).$$
We can estimate
\begin{equation*}
\begin{split}
    \lVert u-f \rVert_{C^{\ell,\alpha}} &\leq \lVert u-v \rVert_{C^{\ell,\alpha}} +\lVert v-f_\delta \rVert_{C^{\ell,\alpha}} + \lVert (1-\delta)(\bar w-f)\rVert_{C^{\ell,\alpha}}\\
    &\leq \|\eta \psi_\delta\|_{C^{\ell,\alpha}}+ \lVert \varphi_\iota*(\eta f_\delta)-\eta f_\delta\rVert_{C^{\ell,\alpha}} +(1-\delta)\eps.
    \end{split}
\end{equation*}
Hence it follows from \eqref{e:normapsi} that for $\delta$ sufficiently small and $\iota$ sufficiently large, $u\in X_\varepsilon$.
Moreover for $\delta$ sufficiently small and $\iota$ sufficiently large we have $u\in \cB(\bar w,\bar\rho/2)$. Indeed, 
\begin{equation*}
\begin{split}
    \lVert u-\bar w \rVert_{C^{\ell,\alpha}}&\leq \lVert u-v \rVert_{C^{\ell,\alpha}} +\lVert v-f_\delta \rVert_{C^{\ell,\alpha}} + \lVert \delta(\bar w-f)\rVert_{C^{{\ell,\alpha}}}\\
     &\leq\|\eta \psi_\delta\|_{C^{{\ell,\alpha}}}+ \lVert \varphi_\iota*(\eta f_\delta)-\eta f_\delta\rVert_{C^{{\ell,\alpha}}}+\delta\varepsilon.
\end{split}
\end{equation*}
By assumption there is a sequence $\rho_i<\bar\rho/2$ with $\rho_i\to 0$, such that there exist $w^i\in\cB(u,\rho_i)$ 
 and $M^i:\Omega\times \langle e_{m-k+1},\ldots, e_m \rangle\to \R^{n}$ of class $C^{h,\beta}$ with $\lVert M^i\rVert_{C^{h,\beta}}\leq j$ for which \eqref{eq:biginterior} fails, that is,
$$\pi_\Omega(\mathrm{gr}(M^i)\cap \mathrm{gr}( w^i)\cap C(x,r))= B(x, r).$$
This means that for any $y \in B(x,r)$ we find $y' \in \R^k$ such that
$$ \big( y, y', M^i(y,y') \big) = \big(y, \underline w^i(y), \underline{\underline w}^i(y) \big) \in \R^{m-k} \times \R^k \times \R^n,\qquad\text{for any $i\in\N$},$$
where we denote $\underline w^i := ( w_1^i, \dots,  w_k^i)$ and $\underline{\underline w}^i:= ( w_{k+1}^i, \dots,  w_{n+k}^i )$.
In particular, for $y=x_t := x+te_1$ with $t\in [0,r/2]$, we have
\[\big(x_t, \underline w^i(x_t), \underline{\underline w}^i(x_t) \big)
    = \big( x_t, \underline w^i(x_t), M^i(x_t, \underline w^i(x_t)) \big),\]
and comparing the last components, we deduce that for every $t\in [0,r/2]$, we have
\begin{equation}\label{e:M_n=omeg}
     w_{n+k}^i(x_t) = M_n^i(x_t, \underline w^i(x_t))\qquad\text{for all }i\in\N.
\end{equation}
Thanks to Arzelà-Ascoli theorem, we can show that there exists $M\in C^{h,\beta}$ with $\lVert M\rVert_{C^{h,\beta}}\leq j$ such that, up to subsequences,  $\lim_{i\to \infty}\lVert M^i-M\rVert_{C^{h,\beta}}=0$. Indeed, up to subsequences, one can find maps $\mathfrak{M}^i_l$ for $l=0,\ldots,h$ such that $D^l M^i$ converges uniformly to $\mathfrak{M}_l$ on $\overline{\Omega}$, the map $\mathfrak{M}_h$ is of class $C^\beta$ and $\sum_{l=1}^N\lVert \mathfrak{M}_l\rVert_\infty+[\mathfrak{M}_h]_{\beta}\leq j$. The fact that $D\mathfrak{M}_l=\mathfrak{M}_{l+1}$ for $l=0,\ldots,h-1$ is an elementary consequence of Lemma \ref{prop:espansioneChbeta} and of the fact that $\lVert M^i\rVert_{C^{h,\beta}}\leq j$.

In addition, since $\rho_i\to 0$ and  $w^i\in\cB(u,\rho_i)$, we also have that $\lim_{i\to\infty}\lVert w^i-u\rVert_{C^1}=0$ and hence, by continuity of all the functions involved, the fact that $w^i\to u$ and $\underline u=\underline v$, \eqref{e:M_n=omeg} implies that
\begin{equation}
u_{n+k}(x_t) = M_n(x_t, \underline u(x_t))=M_n(x_t, \underline v(x_t)),
    \label{id:pippo}
\end{equation}
for every $t\in [0,r/2]$.
On the other hand, using Lemma \ref{prop:espansioneChbeta} and the fact that $M$ is of class $C^{h,\beta}$ and $\underline{u}$ is smooth, we find constants $c_0,\ldots, c_h$ and a function $c_{h+1}(t)$ with $\|c_{h+1}\|_{L^\infty(0,r_\delta)}<C_{h+1}$ such that for every $t\in [0,r/2]$, it holds
\begin{equation}\label{e:taylor2}
M_n(x_t, \underline v(x_t))=c_0+c_1t+\ldots+c_h \frac{t^h}{h!}+c_{h+1}(t)t^{h+\beta}.    
\end{equation}
Observe that $v$ is smooth, hence we can expand it
\begin{equation}\label{e:taylor1}
    v_{n+k}(x_t)=v_{n+k}(x)+\partial_1v_{n+k}(x)t+\ldots+\frac{\partial_1^{h}v_{n+k}(x)}{h!}t^{h}+\frac{\partial_1^{h+1}v_{n+k}(\zeta)}{(h+1)!}t^{h+1},
\end{equation}
for some $\zeta\in [x,x+te_1]$. 

Now we estimate the size of the added bump.
\begin{equation}
\begin{split}\label{e:psidelta}
   \eta(x_t) \psi_\delta(x_t) =& \left( u_{n+k}-v_{n+k}\right)(x_t)\overset{\eqref{id:pippo}}{=}  M_n(x_t,\underline v(x_t))- v_{n+k}(x_t).
\end{split}
\end{equation}
Combining \eqref{e:taylor2}, \eqref{e:taylor1} and \eqref{e:psidelta} we infer that for every $t\in [0,r/2]$ we have
\begin{equation}\label{e:taylor_estimate}
\begin{split}
    \left| \sum_{\kappa=0}^h ( c_\kappa-\partial_1^{\kappa}v_{n+k}(x) )\frac{t^{\kappa}}{\kappa!}
    + c_{h+1}(t)t^{h+\beta}- \frac{\partial_1^{h+1}v_{n+k}(\zeta)}{(h+1)!}t^{h+1}\right|
    &=|M_n(x_t, \underline v(x_t))- v_{n+k}(x_t)|
    \\
     &\overset{\eqref{e:psidelta}}{=} \lvert  \eta(x_t) \psi_\delta(x_t)\rvert\overset{\eqref{e:normapsi}}{=} \delta^{1+\gamma}t^{\ell+\gamma}.
\end{split}
\end{equation}
As $\gamma>0$, we deduce that
\begin{equation} 
    c_\kappa=\partial_1^{\kappa}v_{n+k}(x) \qquad \textnormal{for all} \quad 0\leq \kappa\leq \ell +\gamma.
\label{eq:c=0}
\end{equation}
Moreover, we infer from  \eqref{e:taylor_estimate} that for any $t\in[0,r/2]$, we have
\begin{equation}\label{e:contradiction}
\begin{split}
   \delta^{1+\gamma}t^{\ell+\gamma}
    &\overset{\eqref{eq:c=0}}{=} \Bigg\lvert \sum_{\kappa=\lfloor\ell+ \gamma+1\rfloor}^h (c_\kappa-\partial_1^{\kappa}v_{n+k}(x) )\frac{t^{\kappa}}{\kappa!} + c_{h+1}(t)t^{h+\beta}- \frac{\partial_1^{h+1}v_{n+k}(\zeta)}{(h+1)!}t^{h+1} \Bigg\rvert,
 \end{split}
\end{equation}
which is a contradiction to $\ell+\alpha<\ell+\gamma<h+\beta$.
\end{proof}

\begin{proof}[Proof of Proposition \ref{p:baire1}]
Let us prove that $P2$ has a winning strategy for the Banach-Mazur game associated to $A$. 

Let us assume that the players $P1$ and $P2$ have played already $\kappa$ moves which are associated to open sets $\cU_1,\ldots, \cU_\kappa$ and $\cV_1,\ldots, \cV_\kappa$ chosen by $P1$ and $P2$ respectively in such a way that 
$$\cV_\kappa\subseteq \cU_\kappa \subseteq \cV_{\kappa-1} \subseteq \cdots\subseteq \cV_1\subseteq \cU_1.$$
The $(\kappa+1)$th move for $P1$ is an open set $\cU_{\kappa+1}\subseteq \cV_\kappa$. Now we describe how to choose properly the set $\cV_{\kappa+1}$. 

Let us fix a dense sequence $\{x_\iota\}_{\iota\in\N}$ in $\Omega'$. First $P2$ picks some $\bar w\in \cU_{\kappa+1}$ and $\bar\rho>0$ such that $\cB(\bar w,\bar\rho)\subseteq \cU_{\kappa+1}$. 
By Proposition \ref{lemma:principale} applied with these choices of $\bar w$ and $\bar\rho$ and with $x=x_{\kappa+1}$, $j=\kappa+1$ we obtain 
 $u^{\kappa+1}\in \cB(\bar w,\bar\rho)$
 and $0<\rho_{\kappa+1}<1/(\kappa+1)$ such that 
\begin{itemize}
    \item[(i)] $\cB(u^{\kappa+1}, \rho_{\kappa+1})\subseteq \cU_{\kappa+1}$;
    \item[(ii)] for every $w\in \cB(u^{\kappa+1},\rho_{\kappa+1})$ and $M:\Omega\times \langle e_{m-k+1},\ldots, e_m \rangle\to \R^{n}$ of class $C^{h,\beta}$ with $\lVert M\rVert_{C^{h,\beta}}\leq \kappa+1$ we have
    $$\pi_\Omega(\mathrm{gr}(M)\cap \mathrm{gr}( w)\cap C(x, r))\neq B(x, r),$$
where $r:=\min\{1/(\kappa+1), \mathrm{dist}(x,\Omega\setminus\Omega')\}$.
\end{itemize} 
Note that with the choice $\cV_{\kappa+1}:=\cB(u^{\kappa+1},\rho_{\kappa+1})$ there exists $w_\infty \in X_\eps$ such that
$\{w_\infty\}=\bigcap_{j\in\N} \cV_j$. Let us show that $ w_\infty\not \in A$. Let $U\subseteq\Omega'$ be an open set and pick $\iota\in\N$ such that $B(x_\iota,1/\iota)$ is compactly contained in $U$. Since $ w_\infty\in \cap_{\kappa\in\N} \cV_\kappa$, in particular $ w_\infty\in \cV_\iota$ and we can deduce that $$\pi_\Omega(\mathrm{gr}(M)\cap \mathrm{gr}( w_\infty)\cap C(x_\iota,1/\iota ))\neq B(x_\iota,1/\iota),$$ for every $M:\Omega\times \langle e_{m-k+1},\ldots, e_m \rangle\to \R^{n}$ with $\lVert M\rVert_{C^{h,\beta}}\leq \iota$. Thanks to the arbitrariness of $U$ and since $\iota$ can be chosen arbitrarily large, we conclude that $ w_\infty\not\in A$. Hence this is a winning strategy for $P2$ and this concludes the proof. 
\end{proof}

\begin{proof}[Proof of Proposition \ref{p:baire30}]
Let $w\in \mathcal{A}$ and let $O$ and $\mathcal{N}$ be as in Definition \ref{def:bordibrutti}. Let $p\in \mathrm{gr}(w)\cap O$.  We claim that there exists a ball $B\subseteq O$ centred at $p$ such that the manifold $\mathcal{N}$ inside the ball $B$ coincides with the graph of a map $N:V\to V^\perp$ of class $C^{h,\beta}$, where $V$ is an $m$-dimensional coordinate plane in $\R^{m+n}=\langle e_1,\ldots, e_{m+n}\rangle$. This is due to the implicit function theorem and to the fact that the tangent of $\mathcal{N}$ at $p$ must be a graph with respect to one of the coordinate planes.

Furthermore, it is also clear that $V$ must contain $\langle e_1,\ldots, e_{m-k}\rangle$. This is due to the fact that otherwise $\mathrm{gr}(w\trace\Omega')$ and $\mathrm{gr} (N)$ would be transversal. In particular $\Omega\subseteq V$. For any $m$-dimensional coordinate plane denote with $A_V$ the subset of $X_\varepsilon$ obtained replacing $\langle e_1,\ldots,e_{m}\rangle $ with $V$ in Definition \ref{insiemeA}. Note that the above discussion implies that
$\mathcal A\subseteq \cup_V A_V$, where the union is taken on the coordinate $m$-dimensional planes in $\R^{m+n}$, and thus $\mathcal A$ is of first category by Proposition \ref{p:baire1}. 
\end{proof}

\section{Proof of the main results}\label{s:main_proofs}

Given $\mathcal{M}, \Gamma$ as in Section \ref{s:intro} and $T \in \AMC(b)$ with $b = \llbracket \Gamma \rrbracket$, we recall that a point $p \in \Gamma$ is a \textit{regular boundary point} for $T$ if there exist a neighborhood $W$ of $p$ and an embedded $m$-dimensional submanifold $\Sigma \subset W \cap \mathcal{M}$  of class $C^{1,s}$ for some $s>0$ and without boundary in $W$, such that $\operatorname{supp}(T) \cap W \subset \Sigma$. 
The set of regular boundary points is denoted by $\Regb(T)$ and its complement in $\Gamma$ will be denoted by $\Singb(T)$.

Let $p \in \Regb(T)$. Up to restrictions of $W$ so that $W \cap \Sigma$ is diffeomorphic to an $m$-dimensional ball, there exists a positive integer $Q$ (called \textit{multiplicity}) such that $T\res W=Q \llbracket \Sigma^{+} \rrbracket+(Q-1) \llbracket \Sigma^{-} \rrbracket,$ where $\Sigma^{+}$ and $\Sigma^{-}$ are the two disjoint regular submanifolds of $W$ divided by $\Gamma \cap W$ and with boundaries $\Gamma$ and $-\Gamma$, respectively.
We define the \textit{density} of a regular boundary point $p$ in $\Gamma \cap W$ as $\Theta(T, p):=Q-1/2$. This definition is equivalent to the definition of density of $T$ at every regular boundary point $p$ as $$
\Theta(T, p):=\lim _{r \rightarrow 0} \frac{\|T\|(B(p,r))}{\omega_m r^m},
$$
where the numerator and the denominator represent respectively the mass of the current in a ball of radius $r$ and the $m$-dimensional volume of an $m$-dimensional ball of radius $r$.
Regular boundary points where $Q=1$ are called \textit{one-sided} boundary points. Regular boundary points where $Q>1$ are called \textit{two-sided}. The main result of \cite{Borda} is that, assuming $\mathcal{M}, \Gamma$ and $T$ as above, $\Regb(T)$ is (open and) dense in $\Gamma$.

Analogously, we say that $p \in \operatorname{supp}(T) \setminus \Gamma$ is an \emph{interior regular point} if there is a positive radius $\overline{r}>0$, a regular embedded submanifold $\Sigma \subset \mathcal{M}$ and a positive integer $Q$ such that $T \res B(x_0,\overline{r})=Q \llbracket \Sigma \rrbracket$.
The set of interior regular points, which is relatively open in $\operatorname{supp}(T) \setminus \Gamma$, is denoted by $\Regi(T)$.

\subsection{Proof of Theorem \ref{t:main}}
Observe that for $N=1$ the conclusion of Theorem \ref{t:main} holds for \emph{every} $b\in\mathcal{B}_\eta$, due to \cite[Theorem 2.1]{Borda}. For $N>1$ and every $i=1,\dots,N$ we consider the set $X_{\eta}(p_i)$ and we define the corresponding set $\mathcal{A}_i$ as in Definition \ref{def:bordibrutti}. By Proposition \ref{p:baire30} we have that $\mathcal{A}_i$ is a set of first category in $X_{\eta}(p_i)$ so that $\mathscr{R}:=\prod_{i=1}^N (X_\eta(p_i)\setminus\mathcal{A}_i)$ is a residual set in ${\mathbf{X}}_\eta$, see \eqref{e:defX}. Since the map $\Psi$ defined in \eqref{e:defPSI} is an isometry, then $\Psi(\mathscr{R})$ is a residual set in $\mathcal{B}_\eta$, see \eqref{e:defB}.
Moreover, for every $i=1,\dots, N$ and for every $u\in X_{\eta}(p_i)\setminus\mathcal{A}_i$ the following property holds: for every open set $W\subset\Phii_i(gr(u)|_{\Omega'_i})\subset\mathcal{M}$ and for every $m$-dimensional submanifold $\mathcal{N}\subset\mathcal{M}$ of class $C^{h,\beta}$ such that $\partial\mathcal{N}\cap W=\emptyset$ we have \begin{equation}\label{e:boh} W\cap\Phii_i(gr(u))\not\subset\mathcal{N}.\end{equation}
Now consider $b\in\Psi(\mathscr{R})$ and assume by contradiction that there exists an area-minimizing integral current $T$ with $\partial T=b$ which does not satisfy the conclusion of Theorem \ref{t:main}. By \cite[Theorem 1.6 and Theorem 2.1]{Borda}, the open and dense set of regular boundary points of $T$ contains at least a two-sided point $p$. By \cite[Theorem 2.1]{Borda} the dense set of regular points in the connected component of $\mathrm{supp}(b)$ containing $p$ consists of two-sided points. This contradicts \eqref{e:boh} because for any two-sided point $p$, then supp$(T)$ must be contained in a $C^{h,\beta}$ submanifold, locally around $p$. \qed

\subsection{Proof of Corollary \ref{t:codim1}}
By \cite[Theorem 9.1]{steinbruechel}, for every area-minimizing integral current $T$ with $\partial T=b$ every point $P\in$ supp$(b)$ is regular. As in the proof of Theorem \ref{t:main}, for every $b\in\Psi(\mathscr{R})$ there are no two-sided regular points, which implies that every point of $b$ has density $1/2$. \qed

\subsection{Proof of Theorem \ref{c:main}}



Consider the subset of $\mathcal{B}_\eta$ of those boundaries admitting more than one minimizer:
\begin{equation}\label{d:nonuniqueminimizers} NU:=\{b \in \mathcal{B}_\eta : \text{ there exist } T^1, T^2 \in \AMC(b) \text{ such that } T^1 \neq T^2\}.\end{equation}
We aim to prove that $NU$ is a set of first category in $\mathcal{B}_\eta$. The following lemma shows that it is sufficient to prove that $\mathcal{B}_\eta\setminus NU$ is dense. A similar strategy is adopted in \cite{genericuniqbranchedtransport}.

\begin{lemma}\label{l:denseimpliesresidual}
There exists a constant $\eta_0= \eta_0(\mathcal{M}) >0$ such that if the parameter $\eta$ in \eqref{e:defeta} is smaller than $\eta_0$ the following property holds: if the set $\mathcal{B}_\eta \setminus NU$ is dense in $(\mathcal{B}_\eta, d)$, then it is residual.
\end{lemma}

\begin{proof}
For every $\m \in \N \setminus \{0\}$, consider the sets 
$$NU_\m:=\{b\in \mathcal{B}_\eta: \text{ there exist } T^1,T^2\in\AMC(b) \text{ with } \mathbb{F}(T^2-T^1)\geq \m^{-1}\}.$$
Since $NU_\m\subset NU$, then $(\mathcal{B}_\eta\setminus NU_\m)\supset(\mathcal{B}_\eta\setminus NU)$ and hence, by assumption, $\mathcal{B}_\eta\setminus NU_\m$ is dense in $\mathcal{B}_\eta$ for every $\m$. Therefore $NU_\m$ has empty interior in $\mathcal{B}_\eta$ for every $\m$. We conclude by proving that $NU_\m$ is closed for every $\m$. 

Fix $\m$ and consider a sequence $b_j$ of elements of $NU_\m$ and let $b$ be such that $d(b_j,b) \to 0$. Since $\mathcal{B}_\eta$ is complete, see Lemma \ref{lemmachiusura}, we deduce that $b\in \mathcal{B}_\eta$. By Lemma \ref{l:flatequivalence} we deduce that $\Flat (b_j-b)\to 0$. Observe that, denoting $u(b_j)=(u^j_1,\dots,u^j_N)$, we have
\begin{equation}\label{e:MB}
    \Mass(b_j)\leq\Mass(b_0\res (\mathcal{M}\setminus\bigcup_{i=1}^N \Phii_i(U_i)))+\sum_{i=1}^N\Mass((\textbf{u}_i^j)_\sharp\llbracket \Omega_i \rrbracket)
    \leq C+\sum_{i=1}^N\text{Lip}(\textbf{u}_i^j)^{m-1}\Haus^{m-1}(\Omega_i),
\end{equation}
where we recall that $\textbf{u}_i^j$ are defined in \eqref{e:defbold}. Therefore the masses of $b_j$ are equibounded because $$\text{Lip} (\textbf{u}_i^j)\leq \text{Lip}(f_i)+\|u_i^j-f_i\|_{C^1}\leq \text{Lip}(f_i)+\varepsilon_i.$$
For every $j\in\N$, take 
$$T_j,\bar T_j\in\AMC(b_j)\quad {\mbox{ with }}\quad \mathbb{F}(T_j-\bar T_j)\geq \m^{-1}.$$
Let $T\in\AMC(b)$ and observe that by \cite[Lemma 3.4]{inauenMarchese}, if the parameter $\eta$ defined in \eqref{e:defeta} is smaller than a constant $\eta_0$ depending only on $\mathcal{M}$, for every $j$ sufficiently large there exists $S_j\in{ \mathscr{I}_{m}(\mathcal{M})}$ such that { $b-b_j=\partial S_j$ and $\Mass(S_j)\leq C\, \Flat(b-b_j)$}. Therefore { a competitor for $T_j$ and for $\bar{T}_j$ is $T-S_j$, so that we can} estimate
\begin{equation}\label{e:MC}
    \Mass(T_j)\leq\Mass(T)+C\, \Flat(b-b_j), \qquad \Mass(\bar{T}_j)\leq\Mass(T)+C\, \Flat(b-b_j)
\end{equation}
and so the masses of $T_j$ and $\bar T_j$ are equibounded. Moreover { we claim that $\text{supp}(T_j)$ and $\text{supp}(\bar T_j)$ are contained in a unique compact set $K\subset \mathcal{M}$ for $j$ sufficiently large. We postpone the proof of this claim for the moment.} Combining \eqref{e:MB} and \eqref{e:MC} we deduce from the compactness theorem \cite[\S 4.2.17]{Federer1996GeometricTheory} that there exist integral currents 
${ T_\infty,\bar T_\infty}\in{\mathscr{I}}_{m}(K),$ such that $\partial T_\infty=\partial \bar{T}_\infty=b$ and, up to subsequences, $\mathbb{F}(T_j-T_\infty)\to 0$, $\mathbb{F}(\bar{T}_j-\bar{T}_\infty)\to 0$. Clearly $\mathbb{F}(\bar{T}_\infty-T_\infty)\geq 1/\m$. By \cite[Theorem 34.5]{Simon1983LecturesTheory}, we have $T_\infty, \bar T_\infty\in\AMC(b)$, hence $b\in NU_\m$.

{  Let us prove the claim only for the currents $T_j$, the proof for $\bar T_j$ being identical. Assume by contradiction that, up to passing to a non-relabeled subsequence, there are points $p_j\in{\textrm{supp}}(T_j)$ such that the set $\{p_j:j\in\N\}$ is not bounded and fix $p\in\Gamma$. 
By the density lower bound \cite[Lemma 2.1]{DS}, there exists $D>0, r>0$ such that
\begin{equation}\label{e:dlb}
\Mass(T_j\res B_r(p_j))\geq D r^m, 
\end{equation}
for every $j\in\N$. We will show that this contradicts the minimality of $T_j$, for $j$ large enough.

By \cite[Lemma 28.5]{Simon1983LecturesTheory}, the slices $\langle T_j,d_p,t\rangle$ of $T_j$ with respect to the distance function $d_p(x):=\text{dist}(x,p)$ satisfy 
\[ \langle T_j,d_p,t\rangle = \partial(T_j\res B_{t}(p))-\partial T_j\res B_{t}(p).\]
In particular, since $\partial T_j$ are all supported in a small tubular neighbourhood of $\Gamma$, there exists $t_0>0$ such that the last addendum coincides with $b_j$ for all $j$, if $t\geq t_0$. 

Since $d_p(p_j)$ is not bounded and the masses of $T_j$ are equibounded, then, by \cite[Lemma 28.5 (1)]{Simon1983LecturesTheory}, there exist $j\in\N$ and $t_0<\rho<d_p(p_j)-r$ such that, denoting $\varepsilon_0$ the parameter in \cite[Lemma 3.4]{inauenMarchese}, $$\Flat(\langle T_j,d_p,\rho\rangle)\leq\Mass(\langle T_j,d_p,\rho\rangle)\leq\min\left\{\varepsilon_0, \frac{D}{2}r^m\right\}.$$

By \cite[Lemma 3.4]{inauenMarchese}, there exists an integral $m$-current $Z_j$ such that $\partial Z_j=\langle T_j,d_p,\rho\rangle=\partial(T_j\res B_{\rho}(p))-b_j$ and moreover
\begin{equation}\label{e:compet}
    \Mass(Z_j)=\Flat(\langle T_j,d_p,\rho\rangle)=\Flat(\partial(T_j\res B_{\rho}(p))-b_j)\leq\frac{D}{2}r^m
\end{equation}

The current $R_j:=T_j\res B_{\rho}(p)-Z_j$ has boundary $b_j$ and, using \eqref{e:dlb} and \eqref{e:compet} we deduce the contradiction $\Mass(R_j)<\Mass(T_j)$.
}
\end{proof}

\begin{proposizione}\label{p:densityfinale}
The set $\mathcal{B}_\eta \setminus NU$ is dense in $\mathcal{B}_\eta$.    
\end{proposizione}

\begin{proof}
    Fix $0<\mu<1$. Let $b\in \mathcal {B}_\eta$ and take $(u_1,\dots,u_N)\in {\mathbf{X}}_\eta$ such that $b = \Psi(u_1,\dots,u_N)$, see \eqref{e:defPSI}. Consider $$(w_1,\dots,w_N):=(1-\mu)(u_1,\dots,u_N)-\mu(f_1,\dots,f_N)$$
    and observe that $(w_1,\dots,w_N)\in{\mathbf{X}}_{(1-\mu)\eta}$.
    By Proposition \ref{p:baire30}, there is $(\tilde w_1,\dots,\tilde w_N) \in {\mathbf{X}}_{(1-\mu)\eta}\setminus\mathcal{A}$ with 
    \begin{equation}\label{e:stima1}
    \sum_{i=1,\dots,N}\lVert \tilde w_i-w_i \rVert_{C^{\ell,\alpha}} < \mu\eta/2.    
    \end{equation}
   Now define $b_\mu:= \Psi(w_1,\dots,w_N)$ and $\tilde b:= \Psi(\tilde w_1,\dots,\tilde w_N)$ and observe that by \eqref{e:stima1}
   \begin{equation}\label{e:stima4}
       d(\tilde b,b)\leq d(\tilde b,b_\mu)+d(b_\mu,b)\leq\mu\eta/2+\mu\eta=3\mu\eta/2.
   \end{equation} Moreover, for every $T\in\AMC(\tilde b)$ there exists an open and dense subset of supp$(\tilde b)$ which points have density $1/2$ for $T$. Indeed, fix such a current $T$ and observe that for every $i=1,\dots, N$ \cite[Theorem 1.6]{Borda} implies that $\Phii_i(gr(\tilde{w_i}\res \Omega'_i))$ contains at least one regular boundary point $q_i$ for $T$. On the other hand, by the same argument used in the proof of Theorem \ref{t:main}, $q_i$ cannot be two-sided and therefore it has density $1/2$. Hence \cite[Theorem 2.1]{Borda} implies that all points in the open dense set of regular boundary points for $T$ in the connected component of supp$(\tilde b)$, which contains $q_i$ have density $1/2$.\\ 

   Now we construct a boundary $\hat b$ which is a local perturbation of $\tilde b$ around every $q_i$, with the property that $\AMC(\hat b)$ is a singleton. From now on we we work in every $\Omega_i$ separately and drop the index $i$. Without loss of generality, and up to choosing a subset
   of $U$, we can assume that the diffeomorphism $\Phii:U\to\Phii(U)\subseteq\mathcal{M}$ is of the form \[ \Phii(z) = (z, \Phi(z)) \in \R^{d} \qquad \text{ for } z\in U \subset \R^{n+m}, \]
   with $\Phi: U \to \R^{d-m-n}$ of class $C^{h,\beta}$.
   Moreover up to rotation and translation, we can assume that
   \begin{itemize}
       \item $\displaystyle (0, \Phi(0))=q$ and $D\Phi(0)=0$,
       \item { $T$ admits a nice parametrization over a domain $\Lambda$. More precisely,} there exist $r>0$, an open set $\Lambda\subset B(0,r) \subset \R^{m}$ containing the origin and a $C^{h, \beta}$ function $F: \Lambda \to \R^n$ with $F(0)=0$ and $DF(0)=0$ such that
       \begin{equation}
           \begin{split}
              \displaystyle &\supp(T) \cap C_{\Lambda}\cap B^d(q,r)=\\
               &\left\{\big(x', x_m, F(x', x_m), \Phi(x', x_m, F(x', x_m)) \big): x' \in \Omega', x_m > \tilde w_1(x') \text{ with } (x',x_m) \in \Lambda \right\}, 
           \end{split}
       \end{equation}

 where $B^d(q,r)$ denotes the $d$-dimensional ball, $C_\Lambda \subset \R^d$ the cylinder above $\Lambda$ and $\tilde w_1$ the first component of $\tilde w$.
   \end{itemize}
     Now consider a non-zero, smooth bump function $\rho: \R^{m-1} \to [0, \infty)$ with $\supp(\rho) \subset \Lambda\cap\Omega'$ and $$\lVert \rho \rVert_{C^{\ell, \alpha}} < \frac{\mu \eta}{4(1+ \lVert F \rVert_{C^{\ell,\alpha}})}.$$ 
     Define $v: \Omega \to \R^{1+n} $ by
     \[ v(x')= (\tilde w_1(x') + \rho(x'), F(x', \tilde w_1(x') + \rho(x'))) \]
     and observe that denoting $\hat b:=\Psi(v)$ we have \begin{equation}\label{e:stima2}
         d(\tilde b, \hat b)<\mu\eta/2 
     \end{equation} and that that $\Phii(gr(v)) \subset \supp(T)$. Combining \eqref{e:stima2} and the fact that $\tilde b\in\mathcal{B}_{(1-\mu/2)\eta}$ we deduce that $\hat b\in\mathcal{B}_\eta$. Moreover by \eqref{e:stima4} and \eqref{e:stima2} we also have 
     \begin{equation}\label{e:stima3}
         d(b,\hat b)<2\mu\eta.
     \end{equation}
    Define the current  
    \begin{align*}
        T':= T \res  \left\{(x', x_m, F(x',x_m), y) \in \R^d: x' \in \Lambda\cap\Omega',\,  \tilde w_1(x') < x_m < v_1(x') \right\}.
    \end{align*}
   We repeat this construction in every connected component on $b$ and call the resulting current $T'_i$.
   Consider the current $$\hat T:=T-\sum_{i=1}^N T'_i.$$ Since $T\in\AMC(\tilde b)$, it follows that $\hat T\in\AMC(\hat b)$. Indeed assuming by contradiction that $S\in\AMC(\hat b)$ satisfies $\Mass(S)<\Mass(\hat T)$, we obtain that $\partial (S+\sum_i T'_i)=\tilde b$ and moreover, as $\supp(T'_i)$ are disjoint and $T'_j$ is the restriction of $T$ to a set, we have 
   $$\Mass\Big(S+\sum_i T'_i \Big)\leq \Mass(S)+\sum_i\Mass(T'_i)
   <\Mass(\hat T)+\sum_i \Mass(T'_i)=\Mass(T),$$ 
    which contradicts the minimality of $T$. We claim that $\hat T$ is the unique element of $\AMC(\hat b)$. The validity of the claim concludes the proof due to \eqref{e:stima3} and the arbitrariness of $\mu$.\\



We show the validity of the claim following  \cite{Morganindiana}. Assume by contradiction that there exists a current $\hat S\in\AMC(\hat b)$ with $\hat S\neq\hat T$. Define $S:=\hat S+ \sum_i T'_i$. By interior regularity, see \cite{DeLellisSpadaro1, DeLellisSpadaro2, DeLellisSpadaro3}, there exists a point $q^i \in\Regi(T)\cap\Regi(S)\cap\supp (\hat b)\setminus \supp(\tilde b)$ in every connected component of $\supp (\hat b)\setminus \supp(\tilde b)$. Since both $T$ and $S$ are smooth minimal surfaces in a neighbourhood of $q^i$ which coincide on $\text{supp}(T'_i)$, the unique continuation principle of \cite[Lemma 7.2]{Morganarma} implies that there exists a neighbourhood of $q^i$ where $\text{supp}(S) = \text{supp}(T).$ By \cite[Theorem 2.1]{Borda} we know that in every connected component of $\supp(\hat T)\setminus \supp(\partial \hat T)$ and of $\supp(\hat S)\setminus \supp(\partial \hat S)$ respectively, the sets of interior regular points $\Regi(\hat T)$ and $\Regi(\hat S)$ are connected. Moreover each of these connected components touch (at least) one of the $q^i$. We therefore conclude that $\supp(\hat S) = \supp(\hat T)$. 

Since the points of $\text{supp}(\hat b) \cap  B^d(q,r)$ are one-sided for $\hat T$, then the multiplicity (and the orientation) of $\hat S$ coincides with that of $\hat T$, which concludes the proof of the claim and of Proposition \ref{p:densityfinale}.
\end{proof}


\begin{proof}[Proof of Theorem \ref{c:main}]
By Proposition \ref{p:densityfinale}, $\mathcal{B}_\eta\setminus NU$ is dense in $\mathcal{B}_\eta$ and by Lemma \ref{l:denseimpliesresidual} it is also residual.
\end{proof}

\begin{osservazione}
    The validity of Theorem \ref{c:main} can be extended to the case in which $\mathcal{B}_\eta$ is replaced by the corresponding space of boundaries of class $C^{\infty}$, where ${\mathbf{X}}_\eta$ is endowed with the classical metric inducing the smooth convergence. The argument for the proof differs only in the final steps. We begin the construction of the boundary $\hat b$ starting from $\tilde b:=b_\mu$ and observe that $\tilde b$ could be an element of $\Psi(\mathcal{A})$, so that the regular boundary points of $T\in\AMC(\tilde b)$ in a fixed connected component of supp$(\tilde b)$ might fail to have density 1/2. Let us fix a point $q$ in a connected component of $\text{supp}(\tilde b)$. 
    
    By \cite[Theorem 3.2]{BonDelPas} the current $T$ can be written as a sum of two area-minimizing currents $T_1+T_2$, with $$\Mass(T)=\Mass(T_1)+\Mass(T_2)\quad \mbox{and}\quad \Mass(\partial T)=\Mass(\partial T_1)+\Mass(\partial T_2),$$ where $T_1$ is multiplicity-one and $q\not\in$ supp$(\partial T_2)$. Since $T_1$ has the same local structure which the current $T$ has in the proof of Theorem \ref{c:main}, with same argument we can define a boundary $\hat b_1$ pushing the connected component of $\tilde b$ containing $q$ inside the support of $T_1$ through a map $\varphi$ and prove that the corresponding current $R:=\hat T_1+T_2$ is a minimizer for $\hat b_1$. 
    Now consider any area-minimizing $S\neq R$ with $\partial S=\hat b_1$. 
    Again by \cite[Theorem 3.2]{BonDelPas} the current $S$ can be written as a sum of two area-minimizing currents $S_1+S_2$, with $$\Mass(S)=\Mass(S_1)+\Mass(S_2)\quad \mbox{and}\quad \Mass(\partial S)=\Mass(\partial S_1)+\Mass(\partial S_2),$$ where $S_1$ is multiplicity-one and $\varphi(q)\not\in$ supp$(\partial S_2)$.
    The unique continuation argument used in the proof of Theorem \ref{c:main} guarantees that $\hat T_1=S_1$.  

    Now define $\check b_1 :=\hat b_1-\partial T_2$, pick a connected component of the boundary $\partial T_2$ and iterate the procedure obtaining a new boundary $\check b_2$. Since such number of connected components is strictly decreasing along the iteration, there exists $M\in\N$ such that $\hat b_{M+1}=0$. Obviously there is a unique area-minimizing current with boundary $\hat b:=\check b_1+ \dots +\, \check b_M$.
    \end{osservazione}


{ 
\begin{osservazione}[Other metrics]
    For the typical boundary $b\in\mathcal{B}_\eta$, the conclusions of Theorem \ref{t:main}, Corollary \ref{t:codim1}, and Theorem \ref{c:main} hold for area-minimizing currents with respect to any metric $m$, provided $\mathcal{M}$ and $m$ are sufficiently regular. For instance, if $\mathcal{M}$ is of class $C^{h+1}$ and $\mathcal{M}$ endowed with the metric $m$ embeds isometrically in some Euclidean space as Lipschitz neighbourhood retract of class $C^{h+1}$, then the diffeomorphism between the two manifolds which is obtained composing one isometric embedding with the inverse of the other preserves $C^{h, \beta}$-submanifolds, and therefore all our arguments of Section \ref{s:main_proofs} apply verbatim.

\end{osservazione}
}

\section{Generic uniqueness with respect to the flat norm}\label{s:flat}

In this section we prove Theorem \ref{t:genericga}. We begin with the following

\begin{lemma}\label{l3:closed}
The space $(\mathcal{R}_C, d_{\,\flat})$ is a nontrivial complete metric space.
\end{lemma}

\begin{proof}
It is sufficient to prove that $\mathcal{R}_C$ is closed, then completeness follows from \cite[\S 4.2.17]{Federer1996GeometricTheory}.
Let $b_j$ be a sequence of elements of $\mathcal{R}_C$ and let $b$ be such that $\mathbb{F}(b_j-b)\to 0$. By the lower semicontinuity of the mass, we have $\Mass(b)\leq C$. For any $j\in\N$, let $T_j \in \AMC(b_j)$. By the isoperimetric inequality, see \cite[\S 4.2.10]{Federer1996GeometricTheory}, we have $\sup\{\mathbb{M}(T_j)\}<\infty$. By \cite[\S 4.2.17]{Federer1996GeometricTheory}, there exists $T\in{\mathscr{I}}_{m}(\R^{m+n})$ such that, up to (non relabeled) subsequences $\mathbb{F}(T_j-T)\to 0$. By the continuity of the boundary operator we have $\partial T=b$ and hence $b\in \mathcal{R}_C$. 
\end{proof}

In analogy with \eqref{d:nonuniqueminimizers}, we consider the following subset of $\mathcal{R}_C$:
$$\mathcal{N}\mathcal{U}_C:=\{b\in \mathcal{R}_C: \text{ there exist } T^1, T^2 \in\AMC(b) \text{ such that } T^1\neq T^2\}.$$ The following lemma is the counterpart of Lemma \ref{l:denseimpliesresidual} for the flat norm.

\begin{lemma}\label{l:flatdenseimpliesresidual}
Assume that the set $\mathcal{R}_C\setminus \mathcal{N}\mathcal{U}_C$ is dense in $\mathcal{R}_C$. Then it is residual.
\end{lemma}

\begin{proof}
For $\m \in \N \setminus \{0\}$, consider the sets 
$$\mathcal{N}\mathcal{U}_C^\m:=\{b\in \mathcal{R}_C: \text{ there exist } T^1, T^2 \in\AMC(b) \text{ with } \mathbb{F}(T^2-T^1)\geq \m^{-1}\}.$$
It suffices to prove that $\mathcal{N}\mathcal{U}_C^\m$ is closed for every $\m$. Consider a sequence $b_j$ of elements of $\mathcal{N}\mathcal{U}_C^\m$ and let $b\in \mathcal{R}_C$ be such that $\mathbb{F}(b_j-b)\to 0$. For every $j\in\N$, take $T^1_j,T^2_j \in \AMC(b_j)$ with $\mathbb{F}(T^2_j-T^1_j)\geq 1/\m.$ As in the proof of Lemma \ref{l3:closed}, we deduce that there exist $T^1,T^2\in{\mathscr{I}}_{m}(\R^{m+n})$ such that $\partial T^1=\partial T^2=b$ and, up to (non relabeled) subsequences, $\mathbb{F}(T^1_j-T^1)\to 0$, $\mathbb{F}(T^2_j-T^2)\to 0$ and $\mathbb{F}(T^2-T^1)\geq 1/\m$. By \cite[Theorem 34.5]{Simon1983LecturesTheory}, we have $T_1, T_2 \in \AMC(b)$, hence $b\in \mathcal{N}\mathcal{U}_C^\m$.
\end{proof}

To prove Theorem \ref{t:genericga} we are left to show that the set of boundaries $b \in \mathcal{R}_C$ for which $\AMC(b)$ is a singleton is dense in the metric space $(\mathcal{R}_C,d_{\,\flat}).$ The proof can be roughly summarized as follows: firstly, we approximate $b \in \mathcal{R}_C$ with an integer polyhedral boundary $b_P\in\mathcal{R}_{C-\delta}$, for some $\delta>0$, see Lemma \ref{l:sharppolyhedral}. Then, we fix $S \in \AMC(b_P)$ and for every connected component of $\Regi(S)$ there exists, by \cite{DeLellisSpadaro1, DeLellisSpadaro2, DeLellisSpadaro3}, an interior regular point $x_i$. We define the current $b':=\partial(S-S\res \bigcup_i B(x_i,r_i))$ where $r_i$ are suitably small radii, so that $b'\in\mathcal{R}_C$ and $\Flat(b-b')$ is small. An argument similar to that used in Proposition \ref{p:densityfinale} proves that $\AMC(b')$ is a singleton. 

\begin{lemma}\label{l:sharppolyhedral}
For any $b \in \mathcal{R}_C$ and $\varepsilon>0$ there exist a $\delta >0$ and $b_P \in \mathcal{R}_{C-\delta}\cap\mathscr{P}_{m-1}(K)$ such that $$\mathbb{F}(b-b_P)\leq \varepsilon.$$
\end{lemma}

\begin{proof} 
Without loss of generality and up to rescaling, we can assume $C=1$. We consider a map $\phi:K\to K$ { with image contained in the relative interior of $K$}, which is $(1-\varepsilon/(4m))$-Lipschitz and $\|\text{Id}-\phi\|_\infty<\varepsilon/2^m$. Consider $\bar b:=\phi_\sharp b$. Applying the homotopy formula as in \eqref{e:homotopy_estimate}, we obtain that 
\begin{equation}\label{e:approx1} 
    \Flat(b-\bar b)
    \leq 2^{m-1} \|Id-\phi\|_\infty \leq \varepsilon/2
\end{equation} 
and 
$$\Mass(\bar b)
\leq \left(1-\frac{\varepsilon}{4m} \right)^{m-1}\Mass(b)
\leq \left(1-\frac{\varepsilon}{2} \right)\Mass(b).$$
Observe in particular that $\bar b\in\mathcal{R}_{C{ (1-\varepsilon/2)}}$ and moreover supp$(\bar b)$
is contained in the relative interior of $K$. We can thus apply \cite[\S 4.2.21]{Federer1996GeometricTheory} to obtain an integer polyhedral current $b_P$ such that
\begin{equation}\label{e:approx2}
    \Flat(b_P-\bar b)
    \leq \varepsilon/2, 
    \quad \partial b_P=0 \quad \mbox{and}\quad 
    \Mass(b_P)\leq(1+\varepsilon/2)\Mass(\bar b),
\end{equation}
deducing from \eqref{e:approx1} and \eqref{e:approx2} that $$\Flat(b_P- b)\leq \varepsilon \quad \mbox{and}\quad \Mass(b_P)\leq(1-\varepsilon^2/4)\Mass(b).$$
In particular $b_P$ satisfies the requirement of the lemma for { $\delta=C\varepsilon^2/4$}.

\end{proof}


\begin{proof}[Proof of Theorem \ref{t:genericga}]
Fix $\varepsilon>0$, $b\in\mathcal{R}_C$ and $b_P$ as in Lemma \ref{l:sharppolyhedral} and consider $S \in \AMC(b_P)$. It is sufficient to prove the theorem assuming $\Regi(S)$ is connected, indeed the same argument can be applied to each connected component of $\Regi(S)$.

Let $x_0 \in \Regi(S)$, so that there exists a positive radius $\overline{r}>0$, a smooth embedded submanifold $\Sigma \subset \R^{m+n}$ and a positive integer $Q$ such that $S \res B(x_0,\overline{r})=Q \llbracket \Sigma \rrbracket$. Fix some positive radius $r$ such that $r < \overline{r}$ and define $$S':= S - S \res B(x_0,r)  \,\, \text{ and }\,\, b':= \partial S'.$$ Note that, since $b_{ P}\in\mathcal{R}_{C-\delta}$, then for $r$ sufficiently small $b' \in \mathcal{R}_C$. Note further that $S' \in \AMC(b')$, which can be proved by the same argument used in the proof of Proposition \ref{p:densityfinale}. Hence we only need to show that $\AMC(b') = \{S'\}$. 

Suppose there exists $S'' \in \AMC(b')$ such that $S'\neq S''$ and denote $\hat S:=S''+ S \res B(x_0,r)$. Observe that since 
$$\mathbb M (S) \leq \mathbb M (\hat S) \leq \mathbb M (S'')+\mathbb M ( S \res B(x_0,r)) \leq \mathbb M (S),$$
then $\hat S \in \AMC(b_P)$.
By the minimality of $S$ one immediately sees that $\text{supp}(\hat S)\supset \text{supp} (S)\cap B(x_0,r)$. By interior regularity, there exists $x_1 \in \partial B(x_0,r) \cap \Regi(S) \cap \Regi(\hat S)$. For a sufficiently small radius $\rho$ we can write $$S\res B(x_1,\rho)= Q_1\llbracket \Sigma_1 \rrbracket \res B(x_1,\rho) \, \text{ and } \, \hat S\res B(x_1,\rho)= Q_2\llbracket \Sigma_2 \rrbracket \res B(x_1,\rho).$$ By the same argument of Lemma \ref{p:densityfinale}, the two submanifolds $\Sigma_1, \Sigma_2$ must coincide locally around $x_1$. Since by \cite[Theorem 2.1]{Borda} $\Regi(S')$ and $\Regi(S'')$ are connected, unique continuation implies that $\Regi(S')=\Regi(S'')$. Since all points of $\partial B(x_0,r)$ have density $Q/2$, then the multiplicity (and the orientation) of $S'$ coincides with that of $S''$, contradicting $S'\neq S''$. This proves that the set of boundaries $b \in \mathcal{R}_C$ for which $\AMC(b)$ is a singleton is dense in $(\mathcal{R}_C,d_{\,\flat})$ and hence, by Lemma \ref{l:flatdenseimpliesresidual}, we conclude the proof of Theorem \ref{t:genericga}.
\end{proof}

The preliminary approximation of Lemma \ref{l:sharppolyhedral} is motivated by the following remark. 

\begin{osservazione}\label{r:emptythm}
Given an integral current $b \in \mathscr{B}_{m-1}(K) \cap \mathscr{I}_{m-1}(K)$ and $S \in \AMC(b)$, it is not possible to conclude $\Regi(S) \neq \emptyset$. { Indeed, it could happen that $\textrm{supp}(b)$ = \textrm{supp}(S),} as the following example shows.
Consider a sequence of positive real numbers $r_j$ such that $\sum_jr_j < \infty$ and a sequence $q_j$ that is dense in $B(0,1)\subset\R^2$. Denote $\rho_j:=\min\{r_j,1-|q_j|\}$, and consider the balls $B(q_j, \rho_j)$ and the $2$-dimensional current defined as $$T:= \sum_{j \in \N}\llbracket B(q_j, \rho_j)\rrbracket.$$ Note $T$ is well defined and has finite mass because $\sum_{j \in \N}\Mass(\llbracket B(q_j, \rho_j)\rrbracket)<\infty$ and moreover 
\begin{equation}\label{e:massabordi}
\Mass(\partial T)=\Mass \Big(\sum_{j \in \N}\partial\llbracket B(q_j, \rho_j)\rrbracket \Big)\leq\sum_{j \in \N}\Mass(\partial\llbracket B(q_j, \rho_j)\rrbracket).    
\end{equation}
Hence, by \cite[Theorem 30.3]{Simon1983LecturesTheory}, $T\in\mathscr{I}_{2}(\overline{B(0,1)})$. Moreover, since the intersection between two circumferences with different centers is $\Haus^1$-null, then the inequality in \eqref{e:massabordi} is an equality and therefore $\partial B(q_j,\rho_j) \subset \supp(\partial T)$ for every $j$ implying that supp$(\partial T)=\overline{B(0,1)}$. Now take $S\in\AMC(\partial T)$. Since $\text{supp}(S)\subset \text{conv} (\text{supp}(\partial T))$, we deduce that $\Regi(S)\subset\text{supp}(S)\setminus\text{supp}(\partial T) =\emptyset$.
\end{osservazione}



\printbibliography

\end{document}

%% file: structure.tex
\usepackage[
nochapters, 
pdfspacing, 
dottedtoc 
]{classicthesis} 
\usepackage{amsopn}
\usepackage[T1]{fontenc} 
\usepackage{multirow}
\usepackage[utf8]{inputenc} 
\usepackage{hhline}
\usepackage{graphicx} 
\graphicspath{{Figures/}} 
\usepackage{indentfirst}
\usepackage{enumitem} 
\usepackage{savesym}
\usepackage{mathrsfs}
\usepackage{geometry}
\usepackage{esint}
\usepackage{longtable}

\usepackage[
    backend=biber,
    style=numeric,
    natbib=false,
    url=false, 
    doi=false,
    eprint=false,
    maxnames=50
]{biblatex}

\usepackage{stmaryrd}

\usepackage{subfig} 

\usepackage{amsmath, amssymb,amsthm,amsfonts} 

\usepackage[english,capitalize]{cleveref}

\usepackage{thmtools}
\usepackage{thm-restate}

\theoremstyle{plain}
\newtheorem{teorema}{Theorem}[section]
\newtheorem{proposizione}[teorema]{Proposition}
\newtheorem{lemma}[teorema]{Lemma}
\newtheorem{corollario}[teorema]{Corollary}
\newtheorem*{theorem*}{Theorem}

\theoremstyle{definition}

\theoremstyle{definition}
\newtheorem{definizione}[teorema]{Definition}

\theoremstyle{remark}
\newtheorem{osservazione}[teorema]{Remark}

\newcommand{\Mass}{\mathbb{M}}
\newcommand{\R}{\mathbb{R}}
\newcommand{\N}{\mathbb{N}}

\newcommand{\Haus}{\mathscr{H}}

\newcommand{\D}{\mathscr{D}}

\newcommand{\supp}{\mathrm{supp}}

\newcommand{\eps}{\varepsilon}
\newcommand{\dV}{d_V\kern-1pt}
\newcommand{\dW}{d_W\kern-1pt}

\newcommand{\trait}[3]{\vrule width #1ex height #2ex depth #3ex}
\newcommand{\trace}{\mathchoice%
  {\mathbin{\trait{.12}{1.2}{.03}\trait{.8}{0.09}{0.03}}}
  {\mathbin{\trait{.12}{1.2}{.03}\trait{.8}{0.09}{0.03}}}
  {\mathbin{\hskip.15ex\trait{.09}{.84}{0.02}\trait{.56}{.07}{.02}}\hskip.15ex}
  {\mathbin{\trait{.07}{.6}{.01}\trait{.4}{.06}{.01}}}}

\newcounter{const}

\newcounter{eps}

\newcommand{\vertiii}[1]{{\left\vert\kern-0.25ex\left\vert\kern-0.25ex\left\vert #1 
    \right\vert\kern-0.25ex\right\vert\kern-0.25ex\right\vert}}

\hypersetup{
colorlinks=true, breaklinks=true,bookmarksnumbered,
urlcolor=webbrown, linkcolor=RoyalBlue, citecolor=webgreen, 
pdftitle={}, 
pdfauthor={\textcopyright}, 
pdfsubject={}, 
pdfkeywords={}, 
pdfcreator={pdfLaTeX}, 
pdfproducer={LaTeX with hyperref and ClassicThesis} 
}

%% file: ref.bib
@article {DS,
    AUTHOR = {Duzaar, Frank and Steffen, Klaus.},
     TITLE = {Optimal interior and boundary regularity for almost minimizers to elliptic variational integrals},
   JOURNAL = {J. Reine Angew. Math.},
  FJOURNAL = {},
    VOLUME = {546},
      YEAR = {2002},
    NUMBER = {},
     PAGES = {73--138},
      ISSN = {},
   MRCLASS = {},
  MRNUMBER = {},
MRREVIEWER = {},
       DOI = {},
       URL = {},
}

@book{Federer1996GeometricTheory,
    AUTHOR = {Federer, Herbert},
     TITLE = {Geometric measure theory},
    SERIES = {Die Grundlehren der mathematischen Wissenschaften, Band 153},
 PUBLISHER = {Springer-Verlag New York Inc., New York},
      YEAR = {1969},
     PAGES = {xiv+676},
   MRCLASS = {28.80 (26.00)},
  MRNUMBER = {0257325},
MRREVIEWER = {J. E. Brothers},
}

@book{Simon1983LecturesTheory,
   AUTHOR = {Simon, Leon},
     TITLE = {Lectures on geometric measure theory},
    SERIES = {Proceedings of the Centre for Mathematical Analysis,
              Australian National University},
    VOLUME = {3},
 PUBLISHER = {Australian National University, Centre for Mathematical
              Analysis, Canberra},
      YEAR = {1983},
     PAGES = {vii+272},
      %ISBN = {0-86784-429-9},
   MRCLASS = {49-01 (28A75 49F20)},
  MRNUMBER = {756417},
MRREVIEWER = {J. S. Joel},
}

@article{BonDelPas,
  AUTHOR = {Bonicatto, Paolo and Del Nin, Giacomo and Pasqualetto, Enrico},
     TITLE = {Decomposition of integral metric currents.},
   JOURNAL = {J. Funct. Anal.},
  FJOURNAL = {J. Funct. Anal.},
    VOLUME = {282:7},
      YEAR = {2022},
     PAGES = {Paper No. 109378},
}

@book {Oxtobi,
    AUTHOR = {Oxtoby, John},
     TITLE = {Measure and category. {A} survey of the analogies between
              topological and measure spaces},
      NOTE = {Graduate Texts in Mathematics, Vol. 2},
 PUBLISHER = {Springer-Verlag, New York-Berlin},
      YEAR = {1971},
     %PAGES = {viii+95},
   MRCLASS = {28-02 (54C50 54H05 26A21)},
  MRNUMBER = {0393403},
MRREVIEWER = {I. Chitescu},
}

@article{hardtSimonBoundary,
  title={Boundary regularity and embedded solutions for the oriented Plateau problem},
  author={Hardt, Robert and Simon, Leon},
  journal={Ann. of Math.},
  volume={110:3},
  number={},
  pages={439--486},
  year={1979},
  publisher={}
}

@article{steinbruechel,
  title={Boundary regularity of minimal oriented hypersurfaces on a manifold},
  author={Steinbr{\"u}chel, Simone},
  journal={ESAIM Control Optim. Calc. Var.},
  volume={28, 52},
  pages={},
  year={2022},
  publisher={EDP Sciences}
}

@article{Rado,
  title={On Plateau’s Problem},
  author={Rad\'o, Tibor},
  journal={Ann. of Math.},
  volume={31:3},
  pages={457--469},
  year={1930},
  }

@article{Douglas,
  title={Solution of the problem of Plateau},
  author={Douglas, Jesse},
  journal={Trans. Amer. Math. Soc.},
  volume={33:1},
  pages={263–321},
  year={1931},
  }

@article{Nitsche,
  title={A new uniqueness theorem for minimal surfaces},
  author={Nitsche, Johannes},
  journal={Arch. Ration. Mech. Anal.},
  volume={52:4},
  pages={319–329},
  year={1973}
}

@book{DierkesbookI,
    AUTHOR = {Dierkes, Ulrich and Hildebrandt, Stefan and Sauvigny, Friedrich},
     TITLE = {Minimal Surfaces},
    SERIES = {Grundlehren der mathematischen Wissenschaften},
 PUBLISHER = {Springer},
      YEAR = {2010},
     PAGES = {xvi+688},
}

@book{Courantbook,
    AUTHOR = {Courant, Richard},
     TITLE = {Dirichlet’s principle, conformal mapping and minimal surfaces},
    PUBLISHER = {Interscience},
      YEAR = {1950},
     PAGES = {330},
}

@article{Nitsche3,
  title={Contours Bounding at Least Three Solutions
of Plateau's Problem},
  author={Nitsche, Johannes},
  journal={Arch. Ration. Mech. Anal.},
  volume={30:1},
  pages={1--11},
  year={1968}
}

@article{Morgancontinuum,
  title={A smooth curve in $\R^4$ bounding a continuum of area minimizing surfaces},
  author={Morgan, Frank},
  journal={Duke Math. J.},
  volume={43:4},
  pages={867–-870},
  year={1976}
}

@article{Morganinventiones,
  title={Almost every curve in $\R^3$ bounds a unique area minimizing surface},
  author={Morgan, Frank},
  journal={Invent. Math.},
  volume={45:3},
  pages={253–-297},
  year={1978}
}

@article{Morganindiana,
  title={Generic Uniqueness for Hypersurfaces Minimizing the Integral of an Elliptic Integrand with Constant Coefficients},
  author={Morgan, Frank},
  journal={Indiana Univ. Math. J.},
  volume={30:1},
  pages={29--45},
  year={1981}
}

@article{Morganarma,
  title={Measures on Spaces of Surfaces},
  author={Morgan, Frank},
  journal={Arch. Ration. Mech. Anal.},
  volume={78:4},
  pages={335--359},
  year={1982}
}

@article{Allardphd,
  title={On boundary regularity for Plateau's problem},
  author={Allard, William},
  journal={Bull. Amer. Math. Soc.},
  volume={75:3},
  pages={522--523},
  year={1969}
}

@article{Allardboundary,
  title={On the first variation of a varifold: boundary behavior},
  author={Allard, William},
  journal={Ann. of Math.},
  volume={101:3},
  pages={418--446},
  year={1975}
}

@misc{Borda,
author={De Lellis, Camillo and De Philippis, Guido and Hirsch, Jonas and Massaccesi, Annalisa},
  title={\normalfont{``On the boundary behavior of mass-minimizing integral currents"}},
howpublished = {Preprint arXiv:1809.09457 to appear in: \textit{Mem. Amer. Math. Soc.}},
}

@misc{Nardulli4president,
author={Nardulli, Stefano and Resende, Reinaldo},
  title={\normalfont{``Density of the boundary regular set of 2d area minimizing currents with arbitrary codimension and
multiplicity"}},
howpublished = {Preprint arXiv: 2204.11947},
}

@misc{genericuniqbranchedtransport,
author={Caldini, Gianmarco and Marchese, Andrea and Steinbr\"uchel, Simone},
  title={\normalfont{``Generic uniqueness of optimal transportation networks"}},
howpublished = {Preprint arXiv:2205.05023},
}

@book{Bigregularity,
    AUTHOR = {Almgren, Frederick},
     TITLE = {Almgren’s Big Regularity Paper},
SERIES = {World Scientific Monograph Series in
Mathematics, 1},
 PUBLISHER = {World Scientific Publishing Co. Inc.},
 YEAR = {2000},
     PAGES = {xv+955},
}

@article{inauenMarchese,
  title={Quantitative minimality of strictly stable extremal submanifolds in a flat neighbourhood},
  author={Inauen, Dominik and Marchese, Andrea},
  journal={J. Funct. Anal.},
  volume={275:6},
  number={},
  pages={1532--1550},
  year={2018},
  publisher={Elsevier}
}

@article{DeLellisSpadaro1,
  title={Regularity of area minimizing currents I: gradient $L^{p}$ estimates},
  author={De Lellis, Camillo and Spadaro, Emanuele},
  journal={Geom. Funct. Anal.},
  volume={24},
  number={},
  pages={1831--1884},
  year={2014},
  publisher={}
}

@article{DeLellisSpadaro2,
  title={Regularity of area minimizing currents II: center manifold},
  author={De Lellis, Camillo and Spadaro, Emanuele},
  journal={Ann. of Math. (2)},
  volume={183:2},
  number={},
  pages={499--575},
  year={2016},
  publisher={}
}

@article{DeLellisSpadaro3,
  title={Regularity of area minimizing currents III: blow-up},
  author={De Lellis, Camillo and Spadaro, Emanuele},
  journal={Ann. of Math. (2)},
  volume={183:2},
  number={},
  pages={577--617},
  year={2016},
  publisher={}
}

@article{MarNev21,
  title={Morse index of multiplicity one min-max minimal hypersurfaces},
  author={Marques, Fernando Codá and Neves, André},
  journal={Adv. Math.},
  volume={378:107527},
  number={},
  pages={},
  year={2021},
  publisher={}
}

@article{MarNevSong,
  title={Equidistribution of minimal hypersurfaces for generic metrics},
  author={Marques, Fernando Codá and Neves, André and Song, Antoine},
  journal={Invent. Math.},
  volume={216:2},
  number={},
  pages={421--443},
  year={2019},
  publisher={}
}

@article{IrieMarNev,
  title={Density of minimal hypersurfaces for
generic metrics},
  author={Irie, Kei and Marques, Fernando Codá and Neves, André},
  journal={Ann. of Math. (2)},
  volume={187:3},
  number={},
  pages={963--972},
  year={2018},
  publisher={}
}

@article{ZhouM1Conj,
  title={On the Multiplicity One Conjecture in min-max theory},
  author={Zhou, Xin},
  journal={Ann. of Math. (2)},
  volume={192:3},
  number={},
  pages={767--820},
  year={2020},
  publisher={}
}

@misc{WangLi22,
author={Li, Yangyang and Wang, Zhihan},
  title={\normalfont{``Minimal hypersurfaces for generic metrics in dimension 8"}},
howpublished = {Preprint arXiv:2205.01047},
}

@misc{ChodoshLioSpolaorARS,
author={Chodosh, Otis and Liokumovich, Yevgeny and Spolaor, Luca},
  title={\normalfont{``Singular behavior and generic regularity of min-max minimal hypersurfaces"}},
howpublished = {In: \textit{Ars Inven. Anal.} (2022), paper No. 2},
}

@misc{ChodoshMantouSchulze,
author={Chodosh, Otis and Mantoulidis, Christos and Schulze, Felix},
  title={\normalfont{``Generic regularity for minimizing hypersurfaces in dimensions 9 and 10"}},
howpublished = {Preprint arXiv:2302.02253},
}

@article{White85HGeneric,
  title={Generic regularity of unoriented  two-dimensional area minimizing surfaces},
  author={White, Brian},
  journal={Ann. of Math.},
  volume={121:3},
  number={},
  pages={595--603},
  year={1985},
  publisher={}
}

@misc{White19HGeneric,
author={White, Brian},
  title={\normalfont{``Generic Transversality of Minimal Submanifolds and Generic Regularity of Two-Dimensional Area-Minimizing Integral Currents"}},
howpublished = {Preprint arXiv:1901.05148},
}
